\documentclass[10pt]{amsart}
\usepackage{pgf,tikz}

\usetikzlibrary{decorations.pathreplacing}

\usepackage{caption}
\usepackage{amssymb}
\usepackage{stmaryrd}
\usepackage{mathrsfs}
\usepackage[all]{xy}
\makeatletter
  \def\@wrindex#1{%
    \protected@write\@indexfile{}%
      {\string\indexentry{#1}{ \S\thesubsection (p.\thepage)}}
    \endgroup
  \@esphack}

\makeatother

\makeindex

\newcommand{\DR}{\mathrm{DR}}
\newcommand{\B}{\mathrm{B}}

\usepackage{rotating}
\usepackage{amscd}
\usepackage{epsfig}

\author{Benjamin Enriquez}
\author{Hidekazu Furusho}

\address{Institut de Recherche Math\'{e}matique Avanc\'{e}e, UMR 7501, 
Universit\'{e} de Strasbourg et CNRS, 7 rue Ren\'{e} Descartes, 67000 Strasbourg, France}
\email{enriquez@math.unistra.fr}

\address{Graduate School of Mathematics, Nagoya University, 
Furo-cho, Chikusa-ku, Nagoya, 464-8602, Japan}
\email{furusho@math.nagoya-u.ac.jp}

\date{March 2, 2022}

\newtheorem{thm}{Theorem}[section]
\newtheorem{lem}[thm]{Lemma}
\newtheorem{lemdef}[thm]{Lemma-Definition}
\newtheorem{cor}[thm]{Corollary}
\newtheorem{prop}[thm]{Proposition}

{\theoremstyle{definition} \newtheorem{rem}[thm]{Remark}}
{\theoremstyle{definition} \newtheorem{defn}[thm]{Definition}}

{\theoremstyle{remark} }

\numberwithin{equation}{subsection}
\numberwithin{figure}{section}
\begin{document}

\baselineskip 16pt 

\title[Equality of stabilizer bitorsors]{The stabilizer bitorsors of the module and algebra harmonic coproducts are equal
}

\begin{abstract}

In earlier work, we constructed a pair of "Betti" and "de Rham" Hopf algebras and a pair of module-coalgebras over this pair, as well as the 
bitorsors related to both structures (which will be called the "module" and "algebra" stabilizer bitorsors). We showed that Racinet's torsor 
constructed out of the double shuffle and regularization relations between multiple zeta values is essentially equal to the "module" stabilizer 
bitorsor, and that the latter is contained in the "algebra" stabilizer bitorsor. In this paper, we show the equality of the "algebra" and 
"module" stabilizer bitorsors.  We reduce the proof to showing the equality of the associated "algebra" and "module" graded Lie algebras. 
The argument for showing this equality involves the relation of the "algebra" Lie algebra with the kernel of a linear map, the expression of 
this linear map as a composition of three linear maps, the relation of one of them with the "module" Lie algebra and the computation of the 
kernel of the other one by discrete topology arguments.
\end{abstract}

\bibliographystyle{amsalpha+}
\maketitle
\setcounter{tocdepth}{2}
{\footnotesize \tableofcontents}

\section{Introduction}

The multiple zeta values (MZVs) are the real numbers defined by 
$$
\zeta(n_1,\ldots,n_s):=
\sum_{k_1>\cdots>k_s>0}1/(k_1^{n_1}\cdots k_s^{n_s})
$$
for $s\geq1$ and $n_1\geq 2$, $n_2,\ldots,n_s\geq1$ (\cite{Zagier}). Two sets of algebraic relations between these numbers are particularly known : 
the double shuffle and regularization relations (\cite{R,IKZ}, see also \cite{Ec} where these relations are given without proof) 
and the associator relations which follow from a combination of \cite{Dr} and \cite{LM}. Both sets of relations give rise to 
$\mathbb Q$-schemes denoted $\mathsf{DMR}$ and $\mathsf{M}$, both 
of which are equipped with torsor structures. These sets of relations are conjectured to be equivalent in Racinet's PhD thesis, 
this being equivalent to the equality of schemes $\mathsf{M}=\mathsf{DMR}$. The unfinished preprint \cite{DT} presents ideas on the geometry 
of moduli spaces $\mathfrak M_{0,4}$ and $\mathfrak M_{0,5}$ as well as how to apply them to a proof of the inclusion 
$\mathsf{M}\subset\mathsf{DMR}$. A proof of this inclusion is given in [F] based on the study of Chen's bar calculus for the 
moduli space $\mathfrak M_{0,5}$.

In the series of papers \cite{EF0,EF1,EF2,EF3}, we give a proof of the inclusion $\mathsf{M}\subset\mathsf{DMR}$ based on the ideas of \cite{DT} 
and make explicit the bitorsor aspects of this torsor inclusion. To this end, we attach to each $\mathbb Q$-algebra $\mathbf k$  
and $\omega\in \{\B,\DR\}$ a Hopf algebra $(\hat{\mathcal W}^\omega,
\hat\Delta^{\mathcal W,\omega})$ and a module-coalgebra 
$(\hat{\mathcal M}^\omega,\hat\Delta^{\mathcal M,\omega})$ over it, by which one understand that the action map 
$\hat{\mathcal W}^\omega\otimes\hat{\mathcal M}^\omega\to\hat{\mathcal M}^\omega$ is compatible with the coproducts (see \cite{EF3}, \S3.9). 
Denote by $\mathbf k$-alg (resp. $\mathbf k$-Hopf, $\mathbf k$-alg-mod,  $\mathbf k$-HAMC), the category of $\mathbf k$-algebras 
(resp. $\mathbf k$-Hopf algebras, pairs $(A,M)$ of a $\mathbf k$-algebra $A$ and a $\mathbf k$-module $M$ over $A$, pairs 
$((A,\Delta_A),(M,\Delta_M))$ of a $\mathbf k$-Hopf algebra $(A,\Delta_A)$ and a $\mathbf k$-module coalgebra $(M,\Delta_M)$ over it).
For $\omega\in\{\B,\DR\}$, set $(\hat{\mathcal W},\hat\Delta^{\mathcal W})^\omega:=
(\hat{\mathcal W}^\omega,\hat\Delta^{\mathcal W,\omega})$ and 
$((\hat{\mathcal W},\hat\Delta^{\mathcal W}),(\hat{\mathcal M},\hat\Delta^{\mathcal M}))^\omega:=
((\hat{\mathcal W}^\omega,\hat\Delta^{\mathcal W,\omega}),(\hat{\mathcal M}^\omega,\hat\Delta^{\mathcal M,\omega}))$. 
One constructs a diagram of bitorsors 
$$
\xymatrix{\mathsf G^{\DR,\B}(\mathbf k)\ar[r] & \mathrm{Iso}_{\mathbf k\operatorname{-alg-mod}}((\hat{\mathcal W},\hat{\mathcal 
M})^{\DR/\B})\ar[r]& 
\mathrm{Iso}_{\mathbf k\operatorname{-alg}}(\hat{\mathcal W}^{\DR/\B})
\\ \mathsf M(\mathbf k) \ar@{^{(}->}[u]
\ar^{\!\!\!\!\!\!\!\!\!\!\!\!\!\!\!\!\!\!\!\!\!\!\!\!\!\!\!\!\!\!\!\!\!\!\!\!\!\!\!\!
\!\!\!\!\!\!\!\!\!\!\!(*)}[r]& 
\mathrm{Iso}_{\mathbf k\operatorname{-HAMC}}(
((\hat{\mathcal W},\hat\Delta^{\mathcal W}),(\hat{\mathcal M},\hat\Delta^{\mathcal M}))^{\DR/\B}
) \ar@{^{(}->}[u]\ar[r]& 
\mathrm{Iso}_{\mathbf k\operatorname{-Hopf}}(
(\hat{\mathcal W},\hat\Delta^{\mathcal W})^{\DR/\B}
) \ar@{^{(}->}[u]}
$$
where $\mathsf G^{\DR,\B}(\mathbf k)$ is the ambient bitorsor of the bitorsor of associators (see \cite{Dr}, \S5 and \cite{EF3}, \S\S2.3 and 
3.6) and where for each category $\mathcal C$ and map $\{\B,\DR\}\to\mathrm{Ob}(\mathcal C)$ denoted $\omega\mapsto X^\omega$, we denote by 
$\mathrm{Iso}_{\mathcal C}(X^{\DR/\B})$ the bitorsor $\mathrm{Iso}_{\mathcal C}(X^\B,X^\DR)$; the construction of the morphism $(*)$ relies 
the geometric interpretations of $(\hat{\mathcal W}^\omega,\hat\Delta^{\mathcal W,\omega})$ and 
$(\hat{\mathcal M}^\omega,\hat\Delta^{\mathcal M,\omega})$, which are based on the ideas of \cite{DT}. The fibered product bitorsors 
$$
\mathsf G^{\DR,\B}(\mathbf k)\times_{\mathrm{Iso}_{\mathbf k\operatorname{-alg}}(\mathcal W^{\DR/\B})}
\mathrm{Iso}_{\mathbf k\operatorname{-Hopf}}((\hat{\mathcal W},\hat\Delta^{\mathcal W})^{\DR/\B})
$$ 
and 
$$
\mathsf G^{\DR,\B}(\mathbf k)\times_{
\mathrm{Iso}_{\mathbf k\operatorname{-alg-mod}}((\hat{\mathcal W},\hat{\mathcal 
M})^{\DR/\B})}
\mathrm{Iso}_{\mathbf k\operatorname{-HAMC}}(
((\hat{\mathcal W},\hat\Delta^{\mathcal W}),(\hat{\mathcal M},\hat\Delta^{\mathcal M}))^{\DR/\B}
)
$$
are shown to be equal to bitorsors denoted respectively
$\mathsf{Stab}(\hat\Delta^{\mathcal W,\DR/\B})(\mathbf k)$ and $\mathsf{Stab}(\hat\Delta^{\mathcal M,\DR/\B})(\mathbf k)$
(see \cite{EF3}, \S3.9). The vertical maps being injective, the above diagram leads to a sequence of inclusions of bitorsors
$$
\mathsf M(\mathbf k)\hookrightarrow \mathsf{Stab}(\hat\Delta^{\mathcal M,\DR/\B})(\mathbf k)
\hookrightarrow \mathsf{Stab}(\hat\Delta^{\mathcal W,\DR/\B})(\mathbf k)
\hookrightarrow\mathsf G^{\DR,\B}(\mathbf k). 
$$
The inclusion $\mathsf M(\mathbf k)\subset\mathsf{DMR}(\mathbf k)$ then follows from the identification 
$\mathsf{DMR}(\mathbf k)=\mathsf G_{\mathrm{quad}}^{\DR,\B}(\mathbf k)\cap \mathsf{Stab}(\hat\Delta^{\mathcal M,\DR/\B})(\mathbf k)$
(based on \cite{EF0}) and the easy inclusion $\mathsf M(\mathbf k)\subset\mathsf G_{\mathrm{quad}}^{\DR,\B}(\mathbf k)$, where 
$\mathsf G_{\mathrm{quad}}^{\DR,\B}(\mathbf k)\subset \mathsf G^{\DR,\B}(\mathbf k)$ is a subbitorsor defined by quadratic conditions
(see \cite{EF3}, \S3.1). 

The main result of the present paper is: 

\begin{thm} (see Theorem \ref{thm:24122021})
The inclusion $\mathsf{Stab}(\hat\Delta^{\mathcal M,\DR/\B})(\mathbf k)
\hookrightarrow \mathsf{Stab}(\hat\Delta^{\mathcal W,\DR/\B})(\mathbf k)$ is an equality of bitorsors. 
\end{thm}

Here is an outline of the proof. One reduces the proof of the equality of these bitorsors to that of the underlying $\mathbb Q$-group 
schemes $\mathsf{Stab}(\hat\Delta^{\mathcal M,\DR})(\mathbf k)$ and $\mathsf{Stab}(\hat\Delta^{\mathcal W,\DR})(\mathbf k)$, and then to that of 
their Lie algebra $\mathfrak{stab}(\hat\Delta^{\mathcal M,\DR})$ and $\mathfrak{stab}(\hat\Delta^{\mathcal W,\DR})$ (see proof of Theorem 
\ref{thm:24122021}). Both Lie algebras are degree completions of semidirect products by a one-dimensional Lie algebra of graded Lie algebras
$\mathfrak{stab}_{\mathfrak{lie}(e_0,e_1)}(\Delta^{\mathcal M})$ and $\mathfrak{stab}_{\mathfrak{lie}(e_0,e_1)}(\Delta^{\mathcal W})$, 
which are stabilizer Lie algebras of linear maps $\Delta^{\mathcal M}$ and $\Delta^{\mathcal W}$ (see see \S\S\ref{sect:2:2:2701}, 
\ref{sect:3:2:2701} and Propositions \ref{prop:main:sect:2}, \ref{lem:015:1711}). This reduces the proof to that of the equality of these graded 
Lie algebras (see proof of Theorem \ref{thm:main}).  

This equality is proved as follows. One first expresses $\mathfrak{stab}_{\mathfrak{lie}(e_0,e_1)}(\Delta^{\mathcal W})$ as the preimage by a 
 Lie algebra morphism $\theta : \mathfrak{lie}(e_0,e_1)\to\mathcal V_0$ of the kernel of a linear map 
$-\cdot\Delta^{\mathcal W} : \mathcal V_0\to\mathrm{Der}_{\Delta^{\mathcal W}}(\mathcal W,\mathcal W^{\otimes2})$ (Lemma \ref{lem:inclusion:3110}). 
One then decomposes $-\cdot\Delta^{\mathcal W}$ as a composition $\mathbf i\circ\mathbf h\circ\mathbf H$, where $\mathbf i$, $\mathbf h$ and 
$\mathbf H$ are linear maps (\S\ref{sect:44:2501}). One shows that $\mathbf i$ is injective (\S\ref{subsect:43:2501}), computes the kernel of 
$\mathbf h$ by discrete topology methods (\S\ref{subsub:comp:ker:h}), and relates $\mathbf H$ with $\Delta^{\mathcal M}$ 
(\S\ref{subsub:comm:square}). This leads to a proof of the inclusion of the kernel of $-\cdot\Delta^{\mathcal W}$ in a space 
$(-\cdot 1_{\mathcal M})^{-1}(\mathcal P(\mathcal M))$ defined in terms of $\Delta^{\mathcal M}$ (Proposition \ref{thm:aux}). Using the 
main result of \cite{EF0}, one relates the preimage by $\theta$ of  $(-\cdot 1_{\mathcal M})^{-1}(\mathcal P(\mathcal M))$ with 
$\mathfrak{stab}_{\mathfrak{lie}(e_0,e_1)}(\Delta^{\mathcal M})$ (Proposition \ref{prop:37:2302}), which combined with the other results 
leads to the announced statement (Theorem \ref{thm:main}).  

\medskip 
{\bf Notation.} For $A$ a $\mathbb Q$-algebra, we denote by $\mathrm{Der}(A)$ the Lie algebra of its derivations. 
If $V,W$ are $\mathbb Q$-vector spaces, we denote by $\mathrm{Hom}_{\mathbb Q\operatorname{-vec}}(V,W)$ the $\mathbb Q$-vector 
space of linear maps from $V$ to $W$. We set $\mathrm{End}_{\mathbb Q\operatorname{-vec}}(V):=\mathrm{Hom}_{\mathbb Q\operatorname{-vec}}(V,V)$. 

If $V$ is a $\mathbb Z$-graded vector space and $d\in\mathbb Z$, we denote by $V[d]$ the degree $d$ component of $V$, so $V=\oplus_{d\in\mathbb 
Z}V[d]$. For $v\in V$, we denote by $v[d]$ the component of $v$ of degree $d$, so $v[d]\in V[d]$ and $v=\sum_{d\in \mathbb Z}v[d]$.

\section{The stabilizer Lie algebra $\mathfrak{stab}(\hat\Delta^{\mathcal W,\DR})$}

This section deals with the Lie algebra $\mathfrak{stab}(\hat\Delta^{\mathcal W,\DR})$ from \cite{EF2}. More precisely, in 
\S\ref{sect:def:der:V:1}, we introduce Lie algebras $(\mathfrak{lie}(e_0,e_1),\langle,\rangle)$ and $(\mathcal V_0,\langle,\rangle)$ 
and a Lie algebra morphism $\theta : \mathfrak{lie}(e_0,e_1)\to\mathcal V_0$. In \S\ref{sect:2:2:2701}, we recall from \cite{EF1} the 
algebra $\mathcal W^\DR$ (here denoted $\mathcal W$) and the action of $\mathcal V_0$ on the vector space 
$\mathrm{Hom}_{\mathbb Q\operatorname{-vec}}(\mathcal W,\mathcal W^{\otimes2})$ ; this leads to a stabilizer Lie subalgebra 
$\mathfrak{stab}_{\mathcal V_0}(h)$ of $\mathcal V_0$ for any $h$ in this vector space. In \S\ref{sect:2:3:2701}, we show that if
$h\in\mathrm{Hom}_{\mathbb Q\operatorname{-alg}}(\mathcal W,\mathcal W^{\otimes2})$, then $\mathfrak{stab}_{\mathcal V_0}(h)$  
may be identified with the kernel of a linear map $\mathcal V_0\to\mathrm{Der}_h(\mathcal W,\mathcal W^{\otimes2})$ where the target is the 
set of $h$-derivations of the algebra $\mathcal W$ with values in $\mathcal W^{\otimes2}$, viewed as a $\mathcal W$-bimodule via $h$. 
In \S\ref{sect:def:tilde:yn}, we show that when $h$ is equal to the algebra harmonic coproduct 
$\Delta^{\mathcal W,\DR}$ from \cite{EF1} (here denoted $\Delta^{\mathcal W}$), then the preimage under $\theta$ of 
$\mathfrak{stab}_{\mathcal V_0}(h)$ is a graded Lie subalgebra of $(\mathfrak{lie}(e_0,e_1),\langle,\rangle)$, 
which after undergoing degree completion and semidirect product with the grading action of $\mathbb Q1$, gives the Lie algebra 
$\mathfrak{stab}(\hat\Delta^{\mathcal W,\DR})$ from 
\cite{EF2} (see Proposition \ref{prop:main:sect:2}).  

\subsection{The ambient Lie algebras $(\mathfrak{lie}(e_0,e_1),\langle,\rangle)$ and $(\mathcal V_0,\langle,\rangle)$}
\label{sect:def:der:V:1}

Denote by $\mathcal V$ the free associative $\mathbb Q$-algebra with generators $e_0,e_1$. 
Let $\mathfrak{lie}(e_0,e_1)\subset \mathcal V$ be the Lie subalgebra generated by $e_0,e_1$
(these objects are respectively denoted $\mathcal V^\DR$ and $\mathfrak f_2$ in \cite{EF1}, \S1.1
when $\mathbf k=\mathbb Q$). Then $\mathcal V$ and $\mathfrak{lie}(e_0,e_1)$ are equipped with compatible 
associative and Lie algebra $\mathbb Z_{\geq0}$-gradings, where $e_0$ and $e_1$ have degree 1. Let $\mathcal V_0$
be the direct sum of components of $\mathcal V$ of positive degree.

For $v\in\mathcal V_0$, let $\mathrm{der}_v^{\mathcal V,(1)}$ be the algebra derivation of $\mathcal V$ such that 
$$
\mathrm{der}^{\mathcal V,(1)}_v : e_0\mapsto [v,e_0], e_1\mapsto 0.
$$
For $v,v'\in\mathcal V_0$, set 
$$
\langle v,v'\rangle:=\mathrm{der}^{\mathcal V,(1)}_v(v')-\mathrm{der}^{\mathcal V,(1)}_{v'}(v)+[v',v].
$$

\begin{lem}\label{lem:2810:a}
(a) $(\mathcal V_0,\langle,\rangle)$ is a $\mathbb Z_{\geq0}$-graded Lie algebra, of which $\mathfrak{lie}(e_0,e_1)$ is a graded Lie subalgebra. 

(b) The map $\mathrm{der}^{\mathcal V,(1)}_- : (\mathcal V_0,\langle,\rangle)\to\mathrm{Der}(\mathcal V)$ is a Lie algebra morphism. 
\end{lem}

\begin{proof}
(a) Let $\hat{\mathcal V}_0$ be the degree completion of $\mathcal V_0$ (it is denoted $(\hat{\mathcal V}^\DR_0)_{\mathbb Q}$
in \cite{EF2}, \S3.5). In \cite{EF2}, Lemma 3.8 (b), the space $\mathfrak{em}^\DR:=\mathbb Q\oplus\hat{\mathcal V}_0$ is equipped with 
a Lie bracket $\langle,\rangle$. One checks that $\hat{\mathcal V}_0$ and $\mathcal V_0$ are preserved by this bracket, and are therefore 
Lie subalgebras of $\mathfrak{em}^\DR$. One also checks the bracket to be graded. 

Let also $\mathfrak{lie}(e_0,e_1)^\wedge$ be the degree completion of $\mathfrak{lie}(e_0,e_1)$ (these spaces are denoted 
$(\mathfrak f_2)_{\mathbb Q}^\wedge$ and $(\mathfrak f_2)_{\mathbb Q}$ in \cite{EF2}, \S3.5). In {\it loc. cit.,} it is proved that 
the space $\mathfrak{g}^\DR:=\mathbb Q\oplus\mathfrak{lie}(e_0,e_1)^\wedge$ is a Lie subalgebra of $\mathfrak{em}^\DR$. 
One checks that $\mathfrak{lie}(e_0,e_1)^\wedge$ and $\mathfrak{lie}(e_0,e_1)$ are Lie subalgebras of $\mathfrak g^\DR$. 

(b) In \cite{EF2}, Lemma 3.9, (c), a Lie algebra morphism denoted $\mathfrak{em}^\DR\to\mathrm{Der}(\hat{\mathcal V}_{\mathbb Q},
\hat{\mathcal V}_{\mathbb Q})$, $(\nu,x)\mapsto (\mathrm{der}_{(\nu,x)}^{\mathcal V,(1),\DR},\mathrm{der}_{(\nu,x)}^{\mathcal V,(10),\DR})$ 
is constructed. By \cite{EF2}, Lemma 3.9 (a), $\mathrm{Der}(\hat{\mathcal V}_{\mathbb Q},\hat{\mathcal V}_{\mathbb Q})$ is a Lie subalgebra 
of $\mathrm{Der}(\hat{\mathcal V}_{\mathbb Q})\times\mathrm{End}_{\mathbb Q\operatorname{-vec}}(\hat{\mathcal V}_{\mathbb Q})$, so that the 
projection on the first 
factor of this product is a Lie algebra morphism. Post-composing the Lie algebra morphism $\mathfrak{em}^\DR\to
\mathrm{Der}(\hat{\mathcal V}_{\mathbb Q},\hat{\mathcal V}_{\mathbb Q})$ with this projection and pre-composing it with the inclusion of 
$\mathcal V_0$, one obtains a Lie algebra morphism $\mathcal V_0\to\mathrm{Der}(\hat{\mathcal V})$, whose image can be shown to be contained in 
the Lie subalgebra $\mathrm{Der}(\mathcal V)$ of $\mathrm{Der}(\hat{\mathcal V})$. The resulting map $\mathcal V_0\to\mathrm{Der}(\mathcal V)$ 
is therefore a Lie algebra morphism, and one checks it to be given by $x\mapsto\mathrm{der}^{\mathcal V,(1)}_x$. 
\end{proof}

Let $\theta : \mathfrak{lie}(e_0,e_1)\to\mathcal V_0$ be the map defined by $\theta(x):=x-(x|e_0)e_0+\sum_{n\geq1}(1/n)
(x|e_0^{n-1}e_1)e_1^n$, where $((-|w))_{w\text{ word in }e_0,e_1}$ is the collection of maps $\mathcal V\to\mathbb Q$ such that 
$x=\sum_{w\text{ word in }e_0,e_1}(x|w)w$. 

\begin{lem}\label{lemma:theta:051121}
The map $\theta$ induces a Lie algebra morphism $\theta : (\mathfrak{lie}(e_0,e_1),\langle,\rangle)\to(\mathcal V_0,\langle,\rangle)$. 
\end{lem}

\begin{proof} It follows from \cite{EF2}, Lemma 3.8 (c), that the map $\theta : \mathfrak g^\DR\to\mathfrak{em}^\DR$
given by $(\nu,x)\mapsto(\nu,\theta(x))$ is a Lie algebra morphism. One checks that this map takes the subspace 
$\mathfrak{lie}(e_0,e_1)$ of the source to the subspace $\mathcal V_0$ of the target, which implies the statement. 
\end{proof}

\subsection{Lie algebra action of $(\mathcal V_0,\langle,\rangle)$ on $\mathrm{Hom}_{\mathbb Q\operatorname{-vec}}(\mathcal W,\mathcal 
W^{\otimes2})$}\label{sect:2:2:2701}

Let $\mathcal W:=\mathbb Q\oplus \mathcal Ve_1$, then $\mathcal W\subset\mathcal V$ is a subalgebra (it is denoted 
$\mathcal W^\DR$ in \cite{EF1}, \S1.1 when $\mathbf k=\mathbb Q$). 

\begin{lem}\label{lem:2810:b}
(a) For $v\in\mathcal V_0$,  $\mathrm{der}^{\mathcal V,(1)}_v$  restricts to a derivation $\mathrm{der}^{\mathcal W,(1)}_v$ of $\mathcal W$. 

(b) The map $(\mathcal V_0,\langle,\rangle)\to\mathrm{Der}(\mathcal W)$, $v\mapsto \mathrm{der}^{\mathcal W,(1)}_v$ is a Lie algebra morphism. 
\end{lem}

\begin{proof} (a) follows from the statement preceding \cite{EF2}, Lemma 3.10; it is an immediate consequence of 
$\mathrm{der}^{\mathcal V,(1)}_v(e_1)=0$. (b) follows from (a) and from Lemma \ref{lem:2810:a}; this can also be derived 
from \cite{EF2}, Lemma 3.10, (b). 
\end{proof}

\begin{lem}\label{lem:2810:c}
The $\mathbb Q$-vector space $\mathrm{Hom}_{\mathbb Q\operatorname{-vec}}(\mathcal W,\mathcal W^{\otimes2})$ is equipped with a  
$(\mathcal V_0,\langle,\rangle)$-module structure, 
the action of $v\in\mathcal V_0$ on $h\in\mathrm{Hom}_{\mathbb Q\operatorname{-vec}}(\mathcal W,\mathcal W^{\otimes2})$ being given by 
\begin{equation}\label{act:on:Delta:1309}
v\cdot h:=(\mathrm{der}^{\mathcal W,(1)}_v\otimes\mathrm{id}+\mathrm{id}\otimes \mathrm{der}^{\mathcal W,(1)}_v)\circ h-h\circ
\mathrm{der}^{\mathcal W,(1)}_v. 
\end{equation} 
For $h\in\mathrm{Hom}_{\mathbb Q\operatorname{-vec}}(\mathcal W,\mathcal W^{\otimes2})$, we denote by $- \cdot h : \mathcal V_0
\to\mathrm{Hom}_{\mathbb Q\operatorname{-vec}}(\mathcal W,\mathcal W^{\otimes2})$ the map $v\mapsto v\cdot h$.  
\end{lem}

\begin{proof} Composing the Lie algebra morphism from Lemma \ref{lem:2810:b} (b), the Lie algebra inclusion 
$\mathrm{Der}(\mathcal W)\subset\mathrm{End}_{\mathbb Q\operatorname{-vec}}(\mathcal W)$, and the Lie algebra morphism 
$\mathrm{End}_{\mathbb Q\operatorname{-vec}}(\mathcal W)\to\mathrm{End}_{\mathbb Q\operatorname{-vec}}(\mathrm{Hom}_{\mathbb 
Q\operatorname{-vec}}(\mathcal W,\mathcal W^{\otimes2}))$
given by $f\mapsto (h\mapsto (f\otimes \mathrm{id}_{\mathcal W}+\mathrm{id}_{\mathcal W}\otimes f)\circ h-h\circ f)$, 
one obtains a Lie algebra morphism $(\mathcal V_0,\langle,\rangle)\to\mathrm{End}_{\mathbb Q\operatorname{-vec}}(\mathrm{Hom}_{\mathbb 
Q\operatorname{-vec}}(\mathcal W,
\mathcal W^{\otimes2}))$, i.e. a $(\mathcal V_0,\langle,\rangle)$-module structure over $\mathrm{Hom}_{\mathbb Q\operatorname{-vec}}(\mathcal W,
\mathcal W^{\otimes2})$, which is given by the announced formula.  
\end{proof}

Recall that if $\varphi : \mathfrak h\to\mathfrak g$ is a Lie algebra morphism and if $V$ is a $\mathfrak g$-module, then 
$V$ is equipped with a $\mathfrak h$-module structure by $y\bullet v:=\varphi(y)\cdot v$, called the pull-back module of the 
latter structure by $\varphi$. In particular, the pull-back of the $(\mathcal V_0,\langle,\rangle)$-module structure from 
Lemma \ref{lem:2810:c} on $\mathrm{Hom}_{\mathbb Q\operatorname{-vec}}(\mathcal W,\mathcal W^{\otimes2})$ by $\theta$ is a 
$\mathfrak{lie}(e_0,e_1)$-module structure on the same vector space. 

Recall that if $\mathfrak g$ is a Lie algebra, the data of a pair $(V,v)$ of a $\mathfrak g$-module $V$ and an element $v\in V$
gives rise to a Lie subalgebra $\mathfrak{stab}_{\mathfrak g}(v)$ of $\mathfrak g$, defined as $\{x\in\mathfrak g|x\cdot v=0\}$. 

\begin{lem}\label{lem:inclusion:3110}
Let $h\in \mathrm{Hom}_{\mathbb Q\operatorname{-vec}}(\mathcal W,\mathcal W^{\otimes2})$. 

(a) The stabilizer Lie subalgebra of $(\mathcal V_0,\langle,\rangle)$ of $h$ 
is $\mathfrak{stab}_{\mathcal V_0}(h)=\{v\in\mathcal V_0|v\cdot h=0\}$.   

(b) One has $\mathfrak{stab}_{\mathfrak{lie}(e_0,e_1)}(h)=\theta^{-1}(\mathfrak{stab}_{\mathcal V_0}(h))$. 
\end{lem}

\begin{proof} (a) 
is a specialization of the definition of a stabilizer Lie subalgebra. (b) is a consequence of the fact 
that if $\varphi : \mathfrak h\to\mathfrak g$ is a Lie algebra morphism, if $V$ is a 
$\mathfrak g$-module and if $v\in V$, then $\mathfrak{stab}_{\mathfrak h}(v)=\varphi^{-1}(\mathfrak{stab}_{\mathfrak g}(v))$. 
\end{proof}

\subsection{Corestriction of $-\cdot h$ to $\mathrm{Der}_{h}(\mathcal W,\mathcal W^{\otimes2})$}\label{sect:2:3:2701}

Assume that $h\in\mathrm{Hom}_{\mathbb Q\operatorname{-alg}}(\mathcal W,\mathcal W^{\otimes2})$. 

\begin{defn}
 $\mathrm{Der}_{h}(\mathcal W,\mathcal W^{\otimes2})$ is the set of derivations of $\mathcal W$ with values in $\mathcal 
 W^{\otimes2}$, viewed as a $\mathcal W$-bimodule using $h$; explicitly, this is the set of 
$\mathbb Q$-linear maps $\delta : \mathcal W\to\mathcal W^{\otimes2}$ such that $\delta(ww')=\delta(w)h(w')+h(w)\delta(w')$. 
\end{defn}

\begin{lem}\label{lemma:2039}
The map $D\mapsto(D\otimes\mathrm{id}+\mathrm{id}\otimes D)\circ h-h\circ D$ defines a linear map $\mathrm{Der}(\mathcal 
W)\to\mathrm{Der}_{h}(\mathcal W,\mathcal W^{\otimes2})$. 
\end{lem}

\begin{proof} This follows from the facts that if $D\in\mathrm{Der}(\mathcal W)$, then 
$h\circ D\in\mathrm{Der}_{h}(\mathcal W,\mathcal W^{\otimes2})$ and $D\otimes\mathrm{id}+\mathrm{id}\otimes 
D\in\mathrm{Der}(\mathcal{W}^{\otimes2})$, and if $D'\in\mathrm{Der}(\mathcal{W}^{\otimes2})$, then $D'\circ 
h\in\mathrm{Der}_{h}(\mathcal{W},\mathcal{W}^{\otimes2})$. 
\end{proof}

\begin{cor}\label{corA}
For $v\in\mathcal V_0$, one has $v\cdot h\in \mathrm{Der}_{h}(\mathcal W,\mathcal W^{\otimes2})$, so the 
map $\mathcal V_0\to\mathrm{Hom}_{\mathbb Q\operatorname{-vec}}(\mathcal W,\mathcal W^{\otimes2})$, $v\mapsto v\cdot h$ admits a 
factorization $\mathcal V_0\to\mathrm{Der}_{h}(\mathcal W,\mathcal W^{\otimes2})\subset
\mathrm{Hom}_{\mathbb Q\operatorname{-vec}}(\mathcal W,\mathcal W^{\otimes2})$.  
\end{cor}

\begin{proof} Follows by application of Lemma \ref{lemma:2039} to $D:=\mathrm{der}_v^{\mathcal{V},(1)}$. 
\end{proof}

\begin{lem}\label{lem:2:10}
If $h\in \mathrm{Hom}_{\mathbb Q\operatorname{-alg}}(\mathcal W,\mathcal W^{\otimes2})$, then  
$\mathfrak{stab}_{\mathcal V_0}(h)=\{v\in\mathcal V_0|v\cdot h=0$ (equality in 
$\mathrm{Der}_{h}(\mathcal W,\mathcal W^{\otimes2})$)$\}$.   
\end{lem}

\begin{proof} Follows from Lemma \ref{lem:inclusion:3110} (a) and Corollary \ref{corA}. 
\end{proof}

\subsection{The stabilizer Lie algebras $\mathfrak{stab}_{\mathcal V_0}(\Delta^{\mathcal{W}})$ and $\mathfrak{stab}_{\mathfrak{lie}(e_0,e_1)}(\Delta^{\mathcal{W}})$}\label{sect:def:tilde:yn}

Define the family of elements $(\tilde y_a)_{a\in\mathbb Z}$ of $\mathcal W$ by  
$$
\tilde y_a:=e_0^{a-1}e_1\text{ for }a>0,\quad 
\tilde y_0:=-1,\quad 
\tilde y_a:=0\text{ for }a<0. 
$$

\begin{lem} (\cite{EF1}, \S1 (2))
There is an element $\Delta^{\mathcal W}\in \mathrm{Hom}_{\mathbb Q\operatorname{-alg}}(\mathcal W,\mathcal W^{\otimes2})$
(denoted $\Delta^{\mathcal W,\DR}$ in \cite{EF1}), uniquely determined by 
\begin{equation}\label{def:Delta:W}
\forall n>0,\quad \Delta^{\mathcal W}(\tilde y_n):=-\sum_{i=0}^n \tilde y_i\otimes \tilde y_{n-i}.  
\end{equation} 
\end{lem}

\begin{proof}
This follows from the fact that the algebra $\mathcal W$ is freely generated by $(\tilde y_a)_{a>0}$.  
\end{proof}

Note that \eqref{def:Delta:W} is also valid for $n=0$. 

\begin{lem}\label{lem:16393110}
(a) The stabilizer Lie subalgebra of $(\mathcal V_0,\langle,\rangle)$ of $\Delta^{\mathcal W}$ 
is $\mathfrak{stab}_{\mathcal V_0}(\Delta^{\mathcal W})=\{v\in\mathcal V_0|v\cdot \Delta^{\mathcal W}=0$ (equality in 
$\mathrm{Der}_{\Delta^{\mathcal W}}(\mathcal W,\mathcal W^{\otimes2})$)$\}$.   

(b) One has $\mathfrak{stab}_{\mathfrak{lie}(e_0,e_1)}(\Delta^{\mathcal W})=\theta^{-1}(\mathfrak{stab}_{\mathcal V_0}(\Delta^{\mathcal W}))$. 
\end{lem}

\begin{proof}
(a) follows by specialization from Lemma \ref{lem:2:10}. (b) follows by specialization from Lemma \ref{lem:inclusion:3110}, (b).  
\end{proof}

\subsection{Relation with the Lie algebra $\mathfrak{stab}(\hat\Delta^{\mathcal{W},\DR})$ from \cite{EF2}}

Recall that the Lie algebra $(\mathfrak{lie}(e_0,e_1),\langle,\rangle)$ is $\mathbb Z_+$-graded. Denote by $(\mathfrak{lie}(e_0,e_1)^\wedge,
\langle,\rangle)$ its graded completion. Both Lie algebras are equipped with an action of the abelian Lie $\mathbb Q1$, 
the element $1\in\mathbb Q$ acting by the grading action (multiplying each graded element by its degree). By \cite{EF2}, Lemma 3.8 (b), the Lie 
algebra $\mathfrak g^\DR$ defined in \cite{EF2}, \S3.5 is equal to the corresponding semidirect product, so that 
\begin{equation}\label{g:pdt:semidirect:0511}
\mathfrak g^\DR=\mathbb Q1 \ltimes (\mathfrak{lie}(e_0,e_1)^\wedge,
\langle,\rangle).
\end{equation}
 
\begin{prop}\label{lem:08:1711}\label{prop:main:sect:2}
(a) The Lie subalgebra $\mathfrak{stab}_{\mathfrak{lie}(e_0,e_1)}(\Delta^{\mathcal W})$ of  $(\mathfrak{lie}(e_0,e_1),
\langle,\rangle)$ is graded. Denote by $\mathfrak{stab}_{\mathfrak{lie}(e_0,e_1)}(\Delta^{\mathcal W})^\wedge$ its graded completion. 

(b) The Lie subalgebra $\mathfrak{stab}_{\mathfrak{lie}(e_0,e_1)}(\Delta^{\mathcal W})^\wedge$  of $(\mathfrak{lie}(e_0,e_1)^\wedge,
\langle,\rangle)$ is stable under the action of the abelian Lie algebra $\mathbb Q1$. 

(c) The Lie subalgebra $\mathfrak{stab}(\hat\Delta^{\mathcal W,\DR})$ of $\mathfrak g^\DR$ from \cite{EF2}, §3.5 is equal to the 
corresponding semidirect product, i.e. 
$$
\mathfrak{stab}(\hat\Delta^{\mathcal W,\DR})=\mathbb Q1\ltimes \mathfrak{stab}_{\mathfrak{lie}(e_0,e_1)}(\Delta^{\mathcal W})^\wedge.
$$
\end{prop}

\begin{proof} (a) follows from the fact that $\mathrm{Hom}_{\mathbb Q\operatorname{-vec}}(\mathcal W,\mathcal W^{\otimes2})$ is 
equipped with a $\mathbb Z$-grading, compatible with the grading of the Lie algebra $\mathfrak{lie}(e_0,e_1)$ and with its action on 
it, and that the element $\Delta^{\mathcal W}\in\mathrm{Hom}_{\mathbb Q\operatorname{-vec}}(\mathcal W,\mathcal W^{\otimes2})$ 
is homogeneous (of degree 0). 
(b) then follows from (a). Lemma 3.12 (d) in \cite{EF2} implies that $\mathfrak{stab}(\hat\Delta^{\mathcal W,\DR})$ is the semidirect 
product of $\mathbb Q1$ with the degree completion of the Lie subalgebra 
$\{x\in\mathfrak{lie}(e_0,e_1) | (^\Gamma\mathrm{der}^{\mathcal W,(1)}_{(0,x)}\otimes\mathrm{id}_{\mathcal W}+\mathrm{id}_{\mathcal W} 
\otimes ^\Gamma\mathrm{der}^{\mathcal W,(1)}_{(0,x)}\otimes\mathrm{id}_{\mathcal W})\circ\Delta^{\mathcal W}=\Delta^{\mathcal W} 
\circ  ^\Gamma\mathrm{der}^{\mathcal W,(1)}_{(0,x)}\otimes\mathrm{id}_{\mathcal W}\}$ 
of $\mathfrak{lie}(e_0,e_1)$. By  the first equation of (3.5.4) in \cite{EF2},
$^\Gamma\mathrm{der}^{\mathcal W,(1)}_{(0,x)}=\mathrm{der}^{\mathcal W,(1)}_{\theta(0,x)}$, 
so the latter Lie algebra is $\{x\in\mathfrak{lie}(e_0,e_1)|(\mathrm{der}^{\mathcal W,(1)}_{\theta(x)}\otimes \mathrm{id}_{\mathcal 
W}+\mathrm{id}_{\mathcal W}\otimes\mathrm{der}^{\mathcal W,(1)}_{\theta(x)})\circ \Delta^{\mathcal W}
-\Delta^{\mathcal W}\circ\mathrm{der}^{\mathcal W,(1)}_{\theta(x)}=0\}$, which by Lemma \ref{lem:16393110}, (a) is equal to 
$\theta^{-1}(\mathfrak{stab}_{\mathcal V_0}(\Delta^{\mathcal W}))$, therefore by Lemma \ref{lem:16393110}, (b) equal to 
$\mathfrak{stab}_{\mathfrak{lie}(e_0,e_1)}(\Delta^{\mathcal W})$. This implies (c). 
\end{proof}

\section{The stabilizer Lie algebra $\mathfrak{stab}(\hat\Delta^{\mathcal M,\DR})$}

This section deals with the Lie algebra  $\mathfrak{stab}(\hat\Delta^{\mathcal M,\DR})$ from \cite{EF2}. In \S\ref{sect:3:1:2701}, we 
recall from \cite{EF1} the $\mathcal W$-module $\mathcal M^\DR$ (here denoted $\mathcal M$) and we construct a Lie algebra action of 
$\mathcal V_0$ on $\mathrm{Hom}_{\mathbb Q\operatorname{-vec}}(\mathcal M,\mathcal M^{\otimes2})$. In \S\ref{sect:3:2:2701}, we recall 
from {\it loc. cit.} the definition of the element $\Delta^{\mathcal M,\DR}$ (henceforth denoted $\Delta^{\mathcal M}$) of this vector space, 
called the module harmonic coproduct. This leads to the construction of the stabilizer Lie algebra 
$\mathfrak{stab}_{\mathfrak{lie}(e_0,e_1)}(\Delta^{\mathcal M})$, a graded Lie subalgebra of $\mathfrak{lie}(e_0,e_1)$. In 
\S\ref{sect:3:3:2701}, we show that the Lie algebra $\mathfrak{stab}(\hat\Delta^{\mathcal M,\DR})$ from \cite{EF2} can be obtained from 
$\mathfrak{stab}_{\mathfrak{lie}(e_0,e_1)}(\Delta^{\mathcal M})$ via degree completion and semidirect product with $\mathbb Q1$ for the grading 
action (Proposition \ref{lem:015:1711}). In \S\ref{sect:3:4:2701}, we recall the relation between 
$\mathfrak{stab}_{\mathfrak{lie}(e_0,e_1)}(\Delta^{\mathcal M})$ and the set of primitive elements of $\mathcal M$ for $\Delta^{\mathcal M}$ 
(Theorem 3.1 in \cite{EF0}).  

\subsection{Lie algebra action of $(\mathcal V_0,\langle,\rangle)$ on $\mathrm{Hom}_{\mathbb Q\operatorname{-vec}}(\mathcal M,
\mathcal M^{\otimes2})$}\label{sect:3:1:2701}

When equipped with the left regular action, $\mathcal V$ may be viewed as a graded module over the graded algebra $\mathcal V$. 
Then $\mathcal Ve_0$ is a graded submodule of this $\mathcal V$-module. We denote by 
$$ 
\mathcal M:=\mathcal V/\mathcal Ve_0 
$$ 
the corresponding quotient graded module. We denote by $1_{\mathcal M}\in\mathcal M$ the image of $1\in\mathcal V$. Then the canonical 
projection $\mathcal V\to\mathcal M$ is the map $v\mapsto v\cdot 1_{\mathcal M}$, which we denote by $-\cdot 1_{\mathcal M}$ 
(the module $\mathcal M$ and its element $1_{\mathcal M}$ are denoted $\mathcal M^\DR$, $1_\DR$ in \cite{EF1}, §1.1). 

For $v\in\mathcal V_0$, define $\mathrm{der}_v^{\mathcal V,(10)}$ to be the $\mathbb Q$-vector space endomorphism of $\mathcal V$ such that 
$$
\forall a\in\mathcal V,\quad \mathrm{der}_v^{\mathcal V,(10)}(a)=\mathrm{der}_v^{\mathcal V,(1)}(a)+av.
$$
For any $v\in\mathcal V_0$, $\mathrm{der}_v^{\mathcal V,(10)}$ preserves the subspace $\mathcal Ve_0$, therefore induces an endomorphism
$\mathrm{der}_v^{\mathcal M,(10)}$ of the  $\mathbb Q$-vector space $\mathcal M$. 

\begin{lem}\label{lem:0311:a}
The maps $(\mathcal V_0,\langle,\rangle)\to\mathrm{End}_{\mathbb Q\operatorname{-vec}}(\mathcal V)$, $v\mapsto \mathrm{der}_v^{\mathcal V,(10)}$ and 
$(\mathcal V_0,\langle,\rangle)\to\mathrm{End}_{\mathbb Q\operatorname{-vec}}(\mathcal M)$, $v\mapsto \mathrm{der}_v^{\mathcal M,(10)}$
are Lie algebra morphisms. 
\end{lem}

\begin{proof} By \cite{EF2}, Lemma 3.9 (c), the map $\mathfrak{em}^\DR\to\mathrm{Der}(\hat{\mathcal V}^\DR_{\mathbb Q},
\hat{\mathcal V}^\DR_{\mathbb Q})$, 
$(\nu,x)\mapsto(\mathrm{der}_{(\nu,v)}^{\mathcal V,\DR,(1)},\mathrm{der}_{(\nu,v)}^{\mathcal V,\DR,(10)})$ is a Lie algebra morphism. 
Precomposing it with the Lie algebra injection $(\mathcal V_0,\langle,\rangle)\subset\mathfrak{em}^\DR$, 
$v\mapsto(0,v)$ and post-composing it with the sequence of Lie algebra morphisms 
$\mathrm{Der}(\hat{\mathcal V}^\DR_{\mathbb Q},\hat{\mathcal V}^\DR_{\mathbb Q})\subset
\mathrm{Der}(\hat{\mathcal V}^\DR_{\mathbb Q})\times\mathrm{End}_{\mathbb Q\operatorname{-vec}}(\hat{\mathcal V}^\DR_{\mathbb Q})
\to\mathrm{End}_{\mathbb Q\operatorname{-vec}}(\hat{\mathcal V}^\DR_{\mathbb Q})$, where the second map is the projection on the 
second factor, one sees that the map $(\mathcal V_0,\langle,\rangle)\to\mathrm{End}_{\mathbb Q\operatorname{-vec}}(\hat{\mathcal V}^\DR_{\mathbb Q})$, 
$v\mapsto\mathrm{der}_{(0,v)}^{\mathcal V,\DR,(10)}$ is a Lie algebra morphism. There is a diagram of Lie algebras
$\mathrm{End}_{\mathbb Q\operatorname{-vec}}(\hat{\mathcal V}^\DR_{\mathbb Q})\supset\oplus_{n\in\mathbb Z}\mathrm{End}(\mathcal V)[n]
\subset\mathrm{End}(\mathcal V)$, where $-[n]$ denotes the part of degree $n$. One checks that  
$\mathrm{der}_{(0,v)}^{\mathcal V,\DR,(10)}$ belongs to the Lie subalgebra $\oplus_{n\in\mathbb Z}\mathrm{End}(\mathcal V)[n]$, 
and that the composition of the resulting morphism $(\mathcal V,\langle,\rangle)\to\oplus_{n\in\mathbb Z}\mathrm{End}(\mathcal V)[n]$
with the inclusion $\oplus_{n\in\mathbb Z}\mathrm{End}(\mathcal V)[n]\subset\mathrm{End}(\mathcal V)$ is the map $v\mapsto\mathrm{der}_v^{\mathcal 
V,(10)}$, which is therefore a Lie algebra morphism. The same argument applies with $\mathcal M$ replacing $\mathcal V$, using \cite{EF2}, 
Lemma 3.10, (b) instead of \cite{EF2}, Lemma 3.9 (c). 
\end{proof}

\begin{lem}\label{lem:0311}
The $\mathbb Q$-vector space $\mathrm{Hom}_{\mathbb Q\operatorname{-vec}}(\mathcal M,\mathcal M^{\otimes2})$ is equipped with a  
$(\mathfrak{lie}(e_0,e_1),\langle,\rangle)$-module structure, 
the action of $v\in\mathfrak{lie}(e_0,e_1)$ on $h\in\mathrm{Hom}_{\mathbb Q\operatorname{-vec}}(\mathcal M,\mathcal M^{\otimes2})$ being given by 
\begin{equation}\label{act:on:Delta:0311}
v*h:=(\mathrm{der}^{\mathcal M,(10)}_{\theta(v)}\otimes\mathrm{id}+\mathrm{id}\otimes \mathrm{der}^{\mathcal M,(10)}_{\theta(v)})\circ h
-h\circ\mathrm{der}^{\mathcal M,(10)}_{\theta(v)}. 
\end{equation} 
\end{lem}

\begin{proof} Composing the Lie algebra morphisms $\theta$ from Lemma \ref{lemma:theta:051121},  
$(\mathcal V_0,\langle,\rangle)\to\mathrm{End}_{\mathbb Q\operatorname{-vec}}(\mathcal M)$
from Lemma \ref{lem:0311:a} and  
$\mathrm{End}_{\mathbb Q\operatorname{-vec}}(\mathcal M)\to\mathrm{End}_{\mathbb Q\operatorname{-vec}}(\mathrm{Hom}_{\mathbb Q\operatorname{-vec}}(\mathcal M,\mathcal M^{\otimes2}))$
given by $f\mapsto (h\mapsto (f\otimes \mathrm{id}_{\mathcal M}+\mathrm{id}_{\mathcal M}\otimes f)\circ h-h\circ f)$, 
one obtains a Lie algebra morphism $(\mathfrak{lie}(e_0,e_1),\langle,\rangle)\to\mathrm{End}_{\mathbb Q\operatorname{-vec}}(\mathrm{Hom}_{\mathbb Q\operatorname{-vec}}(\mathcal M,
\mathcal M^{\otimes2}))$, i.e. a $(\mathfrak{lie}(e_0,e_1),\langle,\rangle)$-module structure over $\mathrm{Hom}_{\mathbb Q\operatorname{-vec}}(\mathcal M,
\mathcal M^{\otimes2})$, which is given by the announced formula.  
\end{proof}

\subsection{The stabilizer Lie algebra $\mathfrak{stab}_{\mathfrak{lie}(e_0,e_1)}(\Delta^{\mathcal M})$}\label{sect:3:2:2701}

The $\mathbb Q$-vector space $\mathcal M$ is a a free $\mathcal W$-module of rank 1 generated by $1_{\mathcal M}$
(see \cite{EF1}, \S1.1). 
It follows that there is an element $\Delta^{\mathcal M}$ in $\mathrm{Hom}_{\mathbb Q\operatorname{-vec}}(\mathcal M,\mathcal M^{\otimes2})$, uniquely 
determined by 
\begin{equation}\label{rel:Delta:M:Delta:W}
\forall w\in\mathcal W, \quad 
\Delta^{\mathcal M}(w\cdot 1_{\mathcal M})=\Delta^{\mathcal W}(w)\cdot 1_{\mathcal M}^{\otimes2}
\end{equation}
(see \cite{EF1}, \S1.2, where this element is denoted $\Delta^{\mathcal M,\DR}$). 

\begin{lem}\label{lem:stab:05112021}
(a) The stabilizer Lie algebra of $\Delta^{\mathcal M}\in\mathrm{Hom}_{\mathbb Q\operatorname{-vec}}(\mathcal M,\mathcal M^{\otimes2})$
for the action \eqref{act:on:Delta:0311} is 
$$
\mathfrak{stab}_{\mathfrak{lie}(e_0,e_1)}(\Delta^{\mathcal M})=\{v\in\mathfrak{lie}(e_0,e_1)|v*\Delta^{\mathcal M}=0\}.   
$$

(b) $\mathfrak{stab}_{\mathfrak{lie}(e_0,e_1)}(\Delta^{\mathcal M})$ is a graded Lie subalgebra of 
$(\mathfrak{lie}(e_0,e_1),\langle,\rangle)$.
\end{lem}

\begin{proof}
(a) follows from definitions. (b) follows from the facts that $\mathcal M$, and therefore 
$\mathrm{Hom}_{\mathbb Q\operatorname{-vec}}(\mathcal M,\mathcal M^{\otimes2})$ is a graded module over  
$(\mathfrak{lie}(e_0,e_1),\langle,\rangle)$, and that the element $\Delta^{\mathcal M}$ of the latter
module is homogeneous (of degree 0). 
\end{proof}

\begin{rem}
As explained in the proof of Lemma \ref{lem:0311}, the map $(v,h)\mapsto 
(\mathrm{der}^{\mathcal M,(10)}_{v}\otimes\mathrm{id}+\mathrm{id}\otimes \mathrm{der}^{\mathcal M,(10)}_{v})\circ h
-h\circ\mathrm{der}^{\mathcal M,(10)}_{v}$ defines an action of the Lie algebra $(\mathcal V_0,\langle,\rangle)$
on the vector space $\mathrm{Hom}_{\mathbb Q\operatorname{-vec}}(\mathcal M,\mathcal M^{\otimes2})$ so that as in Lemma 
\ref{lem:16393110}, one may define the stabilizer Lie subalgebra $\mathfrak{stab}_{\mathcal V_0}(\Delta^{\mathcal M})$ of 
$(\mathcal V_0,\langle,\rangle)$ and show that 
$\mathfrak{stab}_{\mathfrak{lie}(e_0,e_1)}(\Delta^{\mathcal M})$ is its preimage under $\theta$. Oppositely to 
$\mathfrak{stab}_{\mathcal V_0}(\Delta^{\mathcal W})$, the Lie algebra $\mathfrak{stab}_{\mathcal V_0}(\Delta^{\mathcal M})$
does not play a role in the sequel of the paper.     
\end{rem}

\subsection{Relation with the Lie algebra $\mathfrak{stab}(\hat\Delta^{\mathcal M,\DR})$ from \cite{EF2}}\label{sect:3:3:2701}

Recall that $(\mathfrak{lie}(e_0,e_1)^\wedge,\langle,\rangle)$ denotes the degree completion of the $\mathbb Z_+$-graded Lie algebra 
$(\mathfrak{lie}(e_0,e_1),\langle,\rangle)$. Denote by $\mathfrak{stab}_{\mathfrak{lie}(e_0,e_1)}(\Delta^{\mathcal M})^\wedge$
the degree completion of its $\mathbb Z_+$-graded Lie subalgebra 
$\mathfrak{stab}_{\mathfrak{lie}(e_0,e_1)}(\Delta^{\mathcal M})$ (see Lemma \ref{lem:stab:05112021} (b)). 

Recall that $(\mathfrak{lie}(e_0,e_1)^\wedge,\langle,\rangle)$ is equipped with the grading action of the abelian Lie algebra $\mathbb Q1$, and 
the identification of the Lie algebra $\mathfrak g^\DR$ from \cite{EF2}, \S3.5 with the corresponding semidirect
product (see \eqref{g:pdt:semidirect:0511}). The following statement is an analogue of Proposition \ref{prop:main:sect:2}.

\begin{prop}\label{lem:015:1711}
(a) The Lie subalgebra $\mathfrak{stab}_{\mathfrak{lie}(e_0,e_1)}(\Delta^{\mathcal M})^\wedge$  of $(\mathfrak{lie}(e_0,e_1)^\wedge,
\langle,\rangle)$ is stable under the action of the abelian Lie algebra $\mathbb Q1$. 

(b) The Lie algebra isomorphism \eqref{g:pdt:semidirect:0511} restricts to an isomorphism of the Lie subalgebra 
$\mathfrak{stab}(\hat\Delta^{\mathcal M,\DR})$ of $\mathfrak g^\DR$ from \cite{EF2}, §3.5 with the 
corresponding semidirect product, i.e. 
$$
\mathfrak{stab}(\hat\Delta^{\mathcal M,\DR})=\mathbb Q1\ltimes \mathfrak{stab}_{\mathfrak{lie}(e_0,e_1)}(\Delta^{\mathcal M})^\wedge.
$$
\end{prop}

\begin{proof} (a) follows from the fact that  $\mathfrak{stab}_{\mathfrak{lie}(e_0,e_1)}(\Delta^{\mathcal M})^\wedge$  is graded complete. 
Lemma 3.12 (e) in \cite{EF2} implies that $\mathfrak{stab}(\hat\Delta^{\mathcal M,\DR})$ is the semidirect 
product of $\mathbb Q1$ with the degree completion of the Lie subalgebra 
$\{x\in\mathfrak{lie}(e_0,e_1) | (^\Gamma\mathrm{der}^{\mathcal M,(10)}_{(0,x)}\otimes\mathrm{id}_{\mathcal M}+\mathrm{id}_{\mathcal M} 
\otimes\ ^\Gamma\!\mathrm{der}^{\mathcal M,(10)}_{(0,x)}\otimes\mathrm{id}_{\mathcal M})\circ\Delta^{\mathcal M}=\Delta^{\mathcal M} \circ  
\ ^\Gamma\!\mathrm{der}^{\mathcal M,(10)}_{(0,x)}\otimes\mathrm{id}_{\mathcal M}\}$ 
of $\mathfrak{lie}(e_0,e_1)$. By  the second equation of (3.5.4) in \cite{EF2},
$^\Gamma\mathrm{der}^{\mathcal M,(10)}_{(0,x)}=\mathrm{der}^{\mathcal M,(10)}_{\theta(0,x)}$, which combined with Lemma \ref{lem:0311}
implies that this Lie algebra is equal to $\mathfrak{stab}_{\mathfrak{lie}(e_0,e_1)}(\Delta^{\mathcal M})^\wedge$. This implies (b). 
\end{proof}

\subsection{Relation of $\mathfrak{stab}_{\mathfrak{lie}(e_0,e_1)}(\Delta^{\mathcal M})$ with $\mathcal P(\mathcal M)$ (Theorem 3.10 from 
\cite{EF0})}\label{sect:3:4:2701}

\begin{lem}
The set $\mathcal P(\mathcal M):=\{m\in\mathcal M|\Delta^{\mathcal M}(m)=m\otimes 1_{\mathcal M}+1_{\mathcal M}\otimes m\}$
is a graded subspace of $\mathcal M$.  
\end{lem}

\begin{proof} Follows from the fact that $\Delta^{\mathcal M}$ is graded. 
\end{proof}

Then $\mathcal P(\mathcal M)$ is the set of primitive elements of $(\mathcal M,\Delta^{\mathcal M})$, which is a cocommutative coalgebra. One 
has 
\begin{equation}\label{identif:PM:10112021}
\mathcal P(\mathcal M)=\mathrm{Ker}(\Delta^{\mathcal M}-(\mathrm{id}\otimes 1_{\mathcal M}+1_{\mathcal M}\otimes\mathrm{id}) : \mathcal 
M\to\mathcal M^{\otimes2}), 
\end{equation} where the map $\mathrm{id}\otimes 1_{\mathcal M}+1_{\mathcal M}\otimes\mathrm{id} : \mathcal M\to\mathcal M^{\otimes2}$ is 
defined by $m\mapsto m\otimes1_{\mathcal M}+1_{\mathcal M}\otimes m$. 

For $d\geq1$, we define $\mathfrak{lie}(e_0,e_1)[\check d]:=\{x\in\mathfrak{lie}(e_0,e_1)|x[d]=0\}$, where we recall that $x[d]$ is the degree $d$ 
component of an element $x\in\mathfrak{lie}(e_0,e_1)$. Then $\mathfrak{lie}(e_0,e_1)[\check d]$ is the direct sum of all the homogeneous components 
of $\mathfrak{lie}(e_0,e_1)$ of degree $\neq d$. 

\begin{prop}\label{prop:37:2302}
$\mathfrak{stab}_{\mathfrak{lie}(e_0,e_1)}(\Delta^{\mathcal M})$ is the intersection with $\mathfrak{lie}(e_0,e_1)[\check 2]$ of the preimage of 
$\mathcal P(\mathcal M)\subset\mathcal M$ by the composed map $\mathfrak{lie}(e_0,e_1)\stackrel{\theta}{\to}\mathcal V
\stackrel{-\cdot 1_{\mathcal M}}{\to}\mathcal M$, i.e. 
\begin{equation}\label{+1459}
\mathfrak{stab}_{\mathfrak{lie}(e_0,e_1)}(\Delta^{\mathcal M})=\theta^{-1}\Big((-\cdot 1_{\mathcal M})^{-1}(\mathcal P(\mathcal M))\Big)
\cap \mathfrak{lie}(e_0,e_1)[\check 2]. 
\end{equation}
In other terms,  $\mathfrak{stab}_{\mathfrak{lie}(e_0,e_1)}(\Delta^{\mathcal M})=\{x\in\mathfrak{lie}(e_0,e_1)|
x[2]=0$ and $\theta(x)\cdot 1_{\mathcal M}\in\mathcal P(\mathcal M)\}$. 
\end{prop}

\begin{proof} Combining Definitions 3.1 and 3.2 and Theorem 3.10 from \cite{EF0} in the case $\Gamma=\{1\}$, one obtains the equality 
\begin{equation}\label{from:EF0}
\mathfrak{stab}(\Delta_\star)=\{\psi\in\mathfrak{Lib}(X)|(\psi|x_0)=(\psi|x_1)=0,(\psi_\star|y_{2,1})=0,\Delta_\star(\psi_\star)=
\psi_\star\otimes 1+1\otimes\psi_\star\}\oplus\mathbb Qx_0\oplus\mathbb Qx_1,  
\end{equation}
where $\mathfrak{stab}(\Delta_\star)$ and $\mathfrak{Lib}(X)$ are the graded Lie algebras defined in \cite{EF0}, §2.5 and §2.1.1, 
the maps $\psi\mapsto(\psi|x_0)$, $\psi\mapsto(\psi|x_1)$ and $\tilde\psi\mapsto(\tilde\psi|y_{21})$ are defined in \cite{R}, §1.6, 
the map $\mathfrak{Lib}(X)\to\mathbb Q\langle Y\rangle$, $\psi\mapsto\psi_\star$ is defined in \cite{EF0}, (2.5), where
$\mathbb Q\langle Y\rangle$ is the graded algebra defined in \cite{EF0}, §2.2, and $\Delta_\star : \mathbb Q\langle Y\rangle\to\mathbb Q\langle 
Y\rangle^{\otimes2}$ is the graded coproduct defined in \cite{EF0}, §2.2. 

Both $\mathfrak{stab}(\Delta_\star)$ and $\{\psi\in\mathfrak{Lib}(X)|\Delta_\star(\psi_\star)=
\psi_\star\otimes 1+1\otimes\psi_\star\}$ are graded subspaces of $\mathfrak{Lib}(X)$ and \eqref{from:EF0} implies that their components of any 
degree $d\neq 1,2$ are equal. \eqref{from:EF0} also implies that $\mathfrak{stab}(\Delta_\star)[1]=\mathbb Qx_0\oplus\mathbb Qx_1$, while the fact 
that any element of 
$\mathbb Q\langle Y\rangle$ of degree one is primitive for $\Delta_\star$, together with the fact that $\psi\mapsto\psi_\star$ is graded, implies  
$\{\psi\in\mathfrak{Lib}(X)|\Delta_\star(\psi_\star)=
\psi_\star\otimes 1+1\otimes\psi_\star\}[1]=\mathbb Qx_0\oplus\mathbb Qx_1$, therefore the degree 1 components of  $\mathfrak{stab}(\Delta_\star)$ 
and 
$\{\psi\in\mathfrak{lie}(e_0,e_1)|\Delta_\star(\psi_\star)=\psi_\star\otimes 1+1\otimes\psi_\star\}$ are equal. 
Finally, $\mathfrak{Lib}(X)[2]$ is one-dimensional, spanned by $[x_0,x_1]$. One computes $[x_0,x_1]_\star=y_{2,1}-(1/2)(y_{1,1})^2$, which is 
primitive for $\Delta_\star$, therefore 
$\{\psi\in\mathfrak{Lib}(X)|\Delta_\star(\psi_\star)=
\psi_\star\otimes 1+1\otimes\psi_\star\}[2]=\mathbb Q\cdot [x_0,x_1]$. On the other hand, $[x_0,x_1]_\star=y_{2,1}-(1/2)(y_{1,1})^2$ also implies 
$([x_0,x_1]_\star|y_{2,1})=1$, which together with \eqref{from:EF0} implies $\mathfrak{stab}(\Delta_\star)[2]=0$.  Putting together these results, 
one obtains 
\begin{equation}\label{result:of:work}
\mathfrak{stab}(\Delta_\star)=\{\psi\in\mathfrak{Lib}(X)|\psi[2]=0\text{ and } \Delta_\star(\psi_\star)=
\psi_\star\otimes 1+1\otimes\psi_\star\}. 
\end{equation}
The map $x_0\mapsto e_0,x_1\mapsto -e_1$ defines a Lie algebra isomorphism $\mathfrak{Lib}(X)\to\mathfrak{lie}(e_0,e_1)$. 
The map taking, for any $r\geq 0$ and $(n_1,\ldots,n_r)\in\mathbb Z_{\geq 1}^r$, the element  $y_{n_1,1}\cdots y_{n_r,1}\in \mathbb Q\langle Y\rangle$ to 
$(-1)^re_0^{n_1-1}e_1\cdots e_0^{n_r-1}e_1\cdot 1_{\mathcal M}\in\mathcal M$ defines a vector space isomorphism $\mathbb Q\langle Y\rangle\to \mathcal M$ with $1\mapsto 1_\mathcal M$. 
These isomorphisms intertwine the maps $\mathfrak{Lib}(X)\to\mathbb Q\langle Y\rangle$, $\psi\mapsto\psi_\star$ 
and $\mathfrak{lie}(e_0,e_1)\mapsto\mathcal M$, $x\mapsto \theta(x)\cdot 1_{\mathcal M}$. The latter isomorphism also intertwines the products 
$\Delta_\star : \mathbb Q\langle Y\rangle\to \mathbb Q\langle Y\rangle^{\otimes2}$ and $\Delta^{\mathcal M} : \mathcal M\to\mathcal M^{\otimes2}$. 
All this implies that the images by the isomorphism $\mathfrak{Lib}(X)\to\mathfrak{lie}(e_0,e_1)$ of the subspaces $\mathfrak{stab}(\Delta_\star)$, 
$\{\psi\in\mathfrak{Lib}(X)|\psi[2]=0\}$ and  $\{\psi\in\mathfrak{Lib}(X)|\Delta_\star(\psi_\star)=\psi_\star\otimes1+1\otimes\psi_\star\}$ of 
$\mathfrak{Lib}(X)$ are respectively the subspaces $\mathfrak{stab}_{\mathfrak{lie}(e_0,e_1)}(\Delta^{\mathcal M})$, $\mathfrak{lie}(e_0,e_1)[\check 
2]$ and $\{x\in\mathfrak{lie}(e_0,e_1)|\theta(x)\cdot 1_{\mathcal M}\in\mathcal P(\mathcal M)\}$ of $\mathfrak{lie}(e_0,e_1)$. The result then 
follows from applying this isomorphism to \eqref{result:of:work}. 
\end{proof}

\section{Decomposition of the map $\mathcal V_0\to\mathrm{Der}_{\Delta^{\mathcal W}}(\mathcal W,\mathcal W^{\otimes2})$}
\label{sect:4:2501}

In this section, we make explicit a decomposition of the map $\mathcal V_0\to\mathrm{Der}_{\Delta^{\mathcal W}}(\mathcal W,
\mathcal W^{\otimes2})$. Its constituents $\mathbf{H}$, $\mathbf{h}$ and $\mathbf i$  are defined respectively in \S\S\ref{sect:41:2501}, 
\ref{subsect:42:2501} and \ref{subsect:43:2501}. The identification of the map $\mathcal V_0\to\mathrm{Der}_{\Delta^{\mathcal W}}(\mathcal W,
\mathcal W^{\otimes2})$ with $\mathbf i\circ \mathbf{h} \circ \mathbf{H}$ is obtained in \S\ref{sect:44:2501} (Proposition 
\ref{prop:upper:comm:triangle}). 

\subsection{The map $\mathbf{H} : \mathcal V_0\to(\oplus_{k\geq0}\mathcal W^{\otimes2})^{\oplus 2}
\oplus(\oplus_{i\geq0}\mathcal W)$}\label{sect:41:2501}

Since $\mathcal V$ is freely generated by $e_0,e_1$, a basis of $\mathcal V_0$ is 
\begin{equation}\label{basis:of:V}
e_0^{a_0-1}e_1\cdots e_1e_0^{a_l-1},\quad\mathrm{ where }\quad l\geq 0,\quad a_0,\ldots,a_l>0\quad\mathrm{ and }\quad a_0>1\quad\mathrm{ if }
\quad l=0.
\end{equation}

\begin{defn}
(a) For $k>0$, ${{L}}_k,{{R}}_k : \mathcal V_0\to\mathcal W^{\otimes2}$ are the linear maps such that   
\begin{equation}\label{eq:beg:1309}
{{L}}_k(e_0^{a_0-1}e_1\cdots e_1e_0^{a_l-1})
=\delta_{a_l,k}(\Delta^{\mathcal W}-\mathrm{id}\otimes1-1\otimes\mathrm{id})(\tilde y_{a_0}\cdots\tilde y_{a_{l-1}}),
\end{equation}
\begin{equation}\label{eq:end:1309} 
{{R}}_k (e_0^{a_0-1}e_1\cdots e_1e_0^{a_l-1})
=\delta_{a_0,k}(\Delta^{\mathcal W}-\mathrm{id}\otimes1-1\otimes\mathrm{id})(\tilde y_{a_1}\cdots\tilde y_{a_l}),
\end{equation}
for any $l\geq 0$ and $a_0,\ldots,a_l>0$ such that $a_0>1$ if $l=0$.

(b) For $i>0$, ${{M}}_i : \mathcal V_0\to\mathcal W$ is the linear map such that  
\begin{equation}\label{eq:mix:1309}
{{M}}_i(e_0^{a_0-1}e_1\cdots e_1e_0^{a_l-1})=\tilde y_{a_0}\cdots\tilde y_{a_{l-1}}\tilde y_{a_l-i}
-\tilde y_{a_0-i}\tilde y_{a_1}\cdots\tilde y_{a_l}
\end{equation}
for any $l\geq 0$ and $a_0,\ldots,a_l>0$ such that $a_0>1$ if $l=0$; recall that $\tilde y_0=-1$ and $\tilde y_a=0$ for $a<0$. 
\end{defn}

\begin{defn}\label{def:map:1409}
 $\mathbf{H} : \mathcal V_0\to(\oplus_{k\geq0}\mathcal W^{\otimes2})\oplus
(\oplus_{k\geq0}\mathcal W^{\otimes2})\oplus
(\oplus_{i\geq0}\mathcal W)$ is the linear map 
$$
\mathbf{H}:=(\oplus_{k\geq 0}{{L}}_{k+1})\oplus(\oplus_{k\geq 0}{{R}}_{k+1})\oplus(\oplus_{i\geq0}{{M}}_{i+1}).
$$
Explicitly, if $v\in\mathcal V_0$, then $$
\mathbf{H}(v)=(({{L}}_{k+1}(v))_{k\geq 0},({{R}}_{k+1}(v))_{k\geq 0},({{M}}_{i+1}(v))_{i\geq 0})\in (\oplus_{k\geq0}\mathcal W^{\otimes2})\oplus
(\oplus_{k\geq0}\mathcal W^{\otimes2})\oplus(\oplus_{i\geq0}\mathcal W).
$$
\end{defn}

\subsection{The map $\mathbf h : (\oplus_{k\geq0}\mathcal W^{\otimes2})^{\oplus 2}
\oplus(\oplus_{i\geq0}\mathcal W)\to\mathrm{Map}(\mathbb Z_{>0},
\mathcal W^{\otimes2})$}\label{subsect:42:2501}

The set $\mathrm{Map}(\mathbb Z_{>0},\mathcal W^{\otimes2})$ of maps from $\mathbb Z_{>0}$ to $\mathcal W^{\otimes2}$ is equipped with  
the $\mathbb Q$-vector space structure inherited from the vector space structure of $\mathcal W^{\otimes2}$.  
Let $\Delta$ be the element of $\mathrm{Map}(\mathbb Z_{>0},\mathcal W^{\otimes2})$ defined by 
\begin{equation}\label{def:f(n)}
\forall n>0,\quad 
\Delta(n):=-\Delta^{\mathcal W}(\tilde y_n)=\sum_{i=0}^n \tilde y_i\otimes\tilde y_{n-i}\in\mathcal W^{\otimes2}. 
\end{equation}

\begin{lem}
(a) For any $k\geq 0$, there is a unique pair of linear maps $\tilde\ell_k,\tilde r_k  : 
\mathcal W\to\mathrm{Map}(\mathbb Z_{>0},\mathcal W^{\otimes2})$, 
such that 
\begin{equation}\label{12151309}
\forall w\in\mathcal W,\quad\forall n>0,\quad \tilde\ell_k(w)(n):=w\Delta(n+k)
,\quad 
\tilde r_k(w)(n):=\Delta(n+k)w 
\quad \text{(equality in $\mathcal W^{\otimes2}$)}.\end{equation}

(b) For any $i\geq 0$, there is a unique linear map $m_i : 
\mathcal V\to\mathrm{Map}(\mathbb Z_{>0},\mathcal W^{\otimes2})$, 
such that 
\begin{equation}\label{12151309bis}
\forall v\in\mathcal V,\quad\forall n>0,\quad \tilde m_i(v)(n):=v\otimes\tilde y_{n+i}+\tilde y_{n+i}\otimes v
\quad \text{(equality in $\mathcal W^{\otimes2}$)}. 
\end{equation}
\end{lem}

\begin{proof} The follows from the fact that the right-hand sides of the equalities of \eqref{12151309} depend linearly on $w\in\mathcal W$, and the 
right-hand side of the equality in \eqref{12151309bis} depends linearly on $v\in\mathcal V$.   
\end{proof}

\begin{defn}\label{def:map':1409}
 $\mathbf h : (\oplus_{k\geq0}\mathcal W^{\otimes2})\oplus
(\oplus_{k\geq0}\mathcal W^{\otimes2})\oplus
(\oplus_{i\geq0}\mathcal W)
\to\mathrm{Map}(\mathbb Z_{>0},\mathcal W^{\otimes2})$ is the linear map 
$$
\mathbf h:=(\oplus_{k\geq 0}\tilde\ell_k)\oplus(\oplus_{k\geq 0}-\tilde r_k)\oplus(\oplus_{i\geq0}\tilde m_i).
$$
Explicitly, if $(\underline a,\underline b,\underline z)\in(\oplus_{k\geq0}\mathcal W^{\otimes2})\oplus
(\oplus_{k\geq0}\mathcal W^{\otimes2})\oplus
(\oplus_{i\geq0}\mathcal W)$, with $\underline a=(a_k)_{k\geq0}$,  $\underline b=(b_k)_{k\geq0}$,  $\underline z=(z_i)_{i\geq0}$, then $\mathbf h(\underline a,\underline b,\underline z)$ is the element of $\mathrm{Map}(\mathbb Z_{>0},\mathcal W^{\otimes2})$ such that 
\begin{equation}\label{explicit:map'}
\forall n>0, \quad \mathbf h(\underline a,\underline b,\underline z)(n)=\sum_{k\geq 0} (a_k\Delta(n+k)-\Delta(n+k)b_k)+\sum_{i\geq0}
(z_i\otimes \tilde y_{n+i}+\tilde y_{n+i}\otimes z_i), 
\end{equation}
where for $n>0$, $\Delta(n)$ is given by \eqref{def:f(n)}.  
\end{defn}

\subsection{The isomorphism $\mathbf i : \mathrm{Map}(\mathbb Z_{>0},
\mathcal W^{\otimes2})\stackrel{\sim}{\to}\mathrm{Der}_{\Delta^{\mathcal W}}(\mathcal W,\mathcal W^{\otimes2})$}
\label{subsect:43:2501}

\begin{defn}
 $\mathrm{Der}_{\Delta^{\mathcal W}}(\mathcal W,\mathcal W^{\otimes2})$ is the set of derivations of $\mathcal W$ with values in $\mathcal 
 W^{\otimes2}$, viewed as a $\mathcal W$-bimodule using $\Delta^{\mathcal W}$; explicitly, this is the set of 
$\mathbb Q$-linear maps $\delta : \mathcal W\to\mathcal W^{\otimes2}$ such that $\delta(ww')=\delta(w)\Delta^{\mathcal W}(w')+\Delta^{\mathcal 
W}(w)\delta(w')$. 
\end{defn}

\begin{lem}\label{lem46:1401}
(a) For $(\delta_n)_{n>0}\in\mathrm{Map}(\mathbb Z_{>0},\mathcal W^{\otimes2})$, there is a unique element 
$\mathbf i((\delta_n)_{n>0})\in\mathrm{Der}_{\Delta^{\mathcal W}}(\mathcal W,\mathcal W^{\otimes2})$ such that 
for $l\geq 0$, 
$n_1,\ldots,n_l>0$, 
$$
\mathbf i((\delta_n)_{n>0})(\tilde y_{n_1}\cdots \tilde y_{n_l})=\sum_{i=1}^l
\Delta^{\mathcal W}(\tilde y_{n_1}\cdots\tilde y_{n_{i-1}})\delta_{n_i}\Delta^{\mathcal W}(\tilde y_{n_{i+1}}\cdots\tilde y_{n_l}). 
$$

(b) The map $\mathbf i : \mathrm{Map}(\mathbb Z_{>0},\mathcal W^{\otimes2})\to
\mathrm{Der}_{\Delta^{\mathcal W}}(\mathcal W,\mathcal W^{\otimes2})$ is a vector space isomorphism, inverse to the map   
$\mathrm{Der}_{\Delta^{\mathcal W}}(\mathcal W,\mathcal W^{\otimes2})\to\mathrm{Map}(\mathbb Z_{>0},\mathcal W^{\otimes2})$ given 
by $\delta\mapsto (\delta(\tilde y_n))_{n>0}$. 
\end{lem}

\begin{proof} (a) Since $\mathcal W$ is freely generated, as an algebra, by the family $(\tilde y_n)_{n>0}$, a vector space basis of 
$\mathcal W$ is the family $(\tilde y_{n_1}\cdots \tilde y_{n_l})_{l\geq 0, n_1,\ldots,n_l>0}$ (the empty product, 
corresponding to $l=0$, being equal to $1$). It follows that $\mathbf i((\delta_n)_{n>0})$ is well-defined as an element of 
$\mathrm{Hom}_{\mathbb Q\operatorname{-vec}}(\mathcal W,\mathcal W^{\otimes2})$. One checks that $\mathbf i((\delta_n)_{n>0})$ belongs to
$\mathrm{Der}_{\Delta^{\mathcal W}}(\mathcal W,\mathcal W^{\otimes2})$. 

(b) One checks that the map $\mathbf i$ is linear. Let $\mathrm{ev}$ be the map $\delta\mapsto (\delta(\tilde y_n))_{n>0}$. 
Then for $(\delta_n)_{n>0}\in\mathrm{Map}(\mathbb Z_{>0},\mathcal W^{\otimes2})$, 
the element $\mathrm{ev}\circ\mathbf i((\delta_n)_{n>0})$ is the map taking $n>0$ to $\mathbf i((\delta_n)_{n>0})(\delta_n)=\delta_n$ so  
$\mathrm{ev}\circ\mathbf i=\mathrm{id}$. For $\delta\in\mathrm{Der}_{\Delta^{\mathcal W}}(\mathcal W,\mathcal W^{\otimes2})$, 
$\mathbf i\circ\mathrm{ev}(\delta)$ and $\delta$ are elements of $\mathrm{Der}_{\Delta^{\mathcal W}}(\mathcal W,\mathcal W^{\otimes2})$
whose restrictions to the set $\{y_n|n>0\}$, which is a generating set of $\mathcal W$ coincide, therefore they are equal. It follows that 
$\mathbf i\circ\mathrm{ev}=\mathrm{id}$. 
\end{proof}

\subsection{Decomposition of the map $\mathcal V_0\to\mathrm{Der}_{\Delta^{\mathcal W}}(\mathcal W,\mathcal W^{\otimes2})$}
\label{sect:44:2501}

\begin{lem}\label{lem:form:of:der}
If $l\geq 0$, $a_1,\ldots,a_l>0$ with $a_0>1$ if $l=0$ and $v:=e_0^{a_0-1}e_1\cdots e_1e_0^{a_l-1}\in\mathcal V_0$, 
then the derivation $\mathrm{der}_v^{\mathcal W,(1)}$ of $\mathcal W$ is such that 
$$
\forall n>0,\quad \mathrm{der}^{\mathcal W,(1)}_v(\tilde y_n)=\tilde y_{a_0}\cdots\tilde y_{a_{l-1}}\tilde y_{a_l-1+n}
-\tilde y_{a_0-1+n}\tilde y_{a_1}\cdots\tilde y_{a_{l}}.
$$
\end{lem}

\begin{proof} 
In the following computation in $\mathcal V$: 
\begin{align*}
&\mathrm{der}_v^{\mathcal W,(1)}(\tilde y_n)
=\mathrm{der}_v^{\mathcal V,(1)}(\tilde y_n)
=\mathrm{der}_v^{\mathcal V,(1)}(e_0^{n-1}e_1)=[v,e_0^{n-1}]e_1
=ve_0^{n-1}e_1-e_0^{n-1}ve_1
\\&=e_0^{a_0-1}e_1\cdots e_1e_0^{a_l+n-2}e_1-e_0^{a_0+n-2}e_1\cdots e_1e_0^{a_l-1}e_1
=\tilde y_{a_0}\cdots\tilde y_{a_{l-1}}\tilde y_{a_l+n-1}-\tilde y_{a_0+n-1}\tilde y_{a_1}\cdots\tilde y_{a_l},
\end{align*}
the first equality follows from the fact that $\mathrm{der}_v^{\mathcal W,(1)}$ is the restriction of $\mathrm{der}_v^{\mathcal V,(1)}$
to $\mathcal W$, the second and last equalities follow from the definition of $\tilde y_n$ (see \S\ref{sect:def:tilde:yn}), and the third equality 
follows from the definition of $\mathrm{der}_v^{\mathcal V,(1)}$ (\S\ref{sect:def:der:V:1}). As $\mathcal W\subset\mathcal V$, the resulting 
identity holds in $\mathcal W$, as claimed. 
\end{proof}

\begin{lem}\label{lem:der:der:Delta:tilde:yn}
For $l\geq 0$, $a_0,\ldots,a_l>0$ with $a_0>1$ if $l=0$, $v:=e_0^{a_0-1}e_1\cdots e_1e_0^{a_l-1}\in\mathcal V_0$ and for $n>0$, 
there holds (equality in $\mathcal W^{\otimes2}$) 
\begin{align*}
&(\mathrm{der}^{\mathcal W,(1)}_v\otimes\mathrm{id}+\mathrm{id}\otimes
\mathrm{der}^{\mathcal W,(1)}_v)\circ\Delta^{\mathcal W}(\tilde y_n)=\sum_{j\geq0}m_{j}(\tilde y_{a_0}\cdots\tilde y_{a_{l-1}}\tilde y_{a_l-1-j}
-\tilde y_{a_0-1-j}\tilde y_{a_1}\cdots\tilde y_{a_{l}})(\tilde y_n)
\\&+\Big(-\ell_{a_l-1}(\tilde y_{a_0}\cdots\tilde y_{a_{l-1}}\otimes1+1\otimes\tilde y_{a_0}
\cdots\tilde y_{a_{l-1}})+r_{a_0-1}(\tilde y_{a_1}\cdots\tilde y_{a_l}\otimes1+1\otimes\tilde y_{a_1}\cdots\tilde y_{a_l})\Big)(\tilde y_n). 
\end{align*}
\end{lem}

\begin{proof} One computes 
\begin{align}\label{interm:1309}
& (\mathrm{der}^{\mathcal W,(1)}_v\otimes\mathrm{id})\circ\Delta^{\mathcal W}(\tilde y_n)
=-(\mathrm{der}^{\mathcal W,(1)}_v\otimes\mathrm{id})(\sum_{i=0}^n \tilde y_i\otimes\tilde y_{n-i})
\\&\nonumber =-\sum_{i=1}^{n} (\tilde y_{a_0}\cdots\tilde y_{a_{l-1}}\tilde y_{a_l-1+i}
-\tilde y_{a_0-1+i}\tilde y_{a_1}\cdots\tilde y_{a_{l}})\otimes\tilde y_{n-i}
\\&\nonumber =-\sum_{i=1-a_l}^{n} \tilde y_{a_0}\cdots\tilde y_{a_{l-1}}\tilde y_{a_l-1+i}\otimes\tilde y_{n-i}
+\sum_{i=1-a_l}^{0} \tilde y_{a_0}\cdots\tilde y_{a_{l-1}}\tilde y_{a_l-1+i}\otimes\tilde y_{n-i}
\\&\nonumber +\sum_{i=1-a_0}^{n}\tilde y_{a_0-1+i}\tilde y_{a_1}\cdots\tilde y_{a_{l}}\otimes\tilde y_{n-i}
-\sum_{i=1-a_0}^{0}\tilde y_{a_0-1+i}\tilde y_{a_1}\cdots\tilde y_{a_{l}}\otimes\tilde y_{n-i}
\\&\nonumber=\sum_{j=0}^{a_l-1} \tilde y_{a_0}\cdots\tilde y_{a_{l-1}}\tilde y_{a_l-1-j}\otimes\tilde y_{n+j}
-\sum_{i=0}^{a_0-1}\tilde y_{a_0-1-j}\tilde y_{a_1}\cdots\tilde y_{a_{l}}\otimes\tilde y_{n+j}
\\&\nonumber -\sum_{s=0}^{n+a_l-1} \tilde y_{a_0}\cdots\tilde y_{a_{l-1}}\tilde y_s\otimes\tilde y_{n+a_l-1-s}
+\sum_{t=0}^{n+a_0-1}\tilde y_t\tilde y_{a_1}\cdots\tilde y_{a_{l}}\otimes\tilde y_{n+a_0-1-t}
\end{align}
where the first equality follows from \eqref{def:Delta:W}, the 
second equality follows from Lemma \ref{lem:form:of:der}, the third equality follows from the addition of 
cancelling terms, and the last equality follows from the changes of variables $j:=-i$, $s:=a_l-1+i$ and $t:=a_0-1+i$. 
Then 
\begin{align*}
& (\mathrm{der}^{\mathcal W,(1)}_v\otimes\mathrm{id}+\mathrm{id}\otimes
\mathrm{der}^{\mathcal W,(1)}_v)\circ\Delta^{\mathcal W}(\tilde y_n)
=-(\tilde y_{a_0}\cdots\tilde y_{a_{l-1}}\otimes1+1\otimes \tilde y_{a_0}\cdots\tilde y_{a_{l-1}})\Delta(n+a_l-1)
\\&+\Delta(n+a_0-1)(\tilde y_{a_1}\cdots\tilde y_{a_{l}}\otimes1+1\otimes\tilde y_{a_1}\cdots\tilde y_{a_{l}})
\\&+\sum_{j=0}^{a_l-1} (\tilde y_{a_0}\cdots\tilde y_{a_{l-1}}\tilde y_{a_l-1-j}\otimes\tilde y_{n+j}+\tilde y_{n+j}
\otimes\tilde y_{a_0}\cdots\tilde y_{a_{l-1}}\tilde y_{a_l-1-j})
\\&-\sum_{i=0}^{a_0-1}(\tilde y_{a_0-1-j}\tilde y_{a_1}\cdots\tilde y_{a_{l}}\otimes\tilde y_{n+j}
+\tilde y_{n+j}\otimes\tilde y_{a_0-1-j}\tilde y_{a_1}\cdots\tilde y_{a_{l}}). 
\\&=
-\ell_{a_l-1}(\tilde y_{a_0}\cdots\tilde y_{a_{l-1}}\otimes1+1\otimes\tilde y_{a_0}
\cdots\tilde y_{a_{l-1}})(\tilde y_n)
+r_{a_0-1}(\tilde y_{a_1}\cdots\tilde y_{a_{l}}\otimes1+1\otimes\tilde y_{a_1}
\cdots\tilde y_{a_{l}})(\tilde y_n)
\\&+\sum_{j=0}^{a_l-1}m_{j}(\tilde y_{a_0}\cdots\tilde y_{a_{l-1}}\tilde y_{a_l-1-j})(\tilde y_n)
-\sum_{j=0}^{a_0-1}m_{j}(\tilde y_{a_0-1-j}\tilde y_{a_1}\cdots\tilde y_{a_{l}})(\tilde y_n)
\end{align*}
where the first equality follows from the symmetrization of \eqref{interm:1309} with respect to  the exchange of factors of 
$\mathcal W^{\otimes2}$, together with the identification of the last line of this equation with 
$-(\tilde y_{a_0}\cdots\tilde y_{a_{l-1}}\otimes 1)\Delta(n+a_l-1)+\Delta(n+a_0-1)(\tilde y_{a_1}\cdots\tilde y_{a_l}\otimes 1)$, 
and the second equality follows from \eqref{12151309} and \eqref{12151309bis}. The statement then follows from $\tilde y_a=0$ for $a<0$
and the linearity of $m_j$ for $j\geq0$. \end{proof}

\begin{lem}\label{lem:Delta:der:a:tilde:yn}
For $l\geq 0$, $a_0,\ldots,a_l>0$ with $a_0>1$ if $l=0$, $v:=e_0^{a_0-1}e_1\cdots e_1e_0^{a_l-1}\in\mathcal V_0$ and for $n>0$, 
there holds 
$$
\Delta^{\mathcal W}\circ \mathrm{der}^{\mathcal W,(1)}_v(\tilde y_n)
=\Big(-\ell_{a_l-1}(\Delta^{\mathcal W}(\tilde y_{a_0}\cdots\tilde y_{a_{l-1}}))
+r_{a_0-1}(\Delta^{\mathcal W}(\tilde y_{a_1}\cdots\tilde y_{a_{l}}))\Big)(\tilde y_{n})\in\mathcal W^{\otimes2}. 
$$
\end{lem}

\begin{proof}
One computes 
\begin{align*}
&\Delta^{\mathcal W}\circ \mathrm{der}^{\mathcal W,(1)}_v(\tilde y_n)
=\Delta^{\mathcal W}(\tilde y_{a_0}\cdots\tilde y_{a_{l-1}}\tilde y_{a_l-1+n}
-\tilde y_{a_0-1+n}\tilde y_{a_1}\cdots\tilde y_{a_{l}})
 \\&=-\Delta^{\mathcal W}(\tilde y_{a_0}\cdots\tilde y_{a_{l-1}})\Delta(a_l-1+n)
 +\Delta(a_0-1+n)\Delta^{\mathcal W}(\tilde y_{a_1}\cdots\tilde y_{a_{l}}) 
 \\&=\Big(-\ell_{a_l-1}(\Delta^{\mathcal W}(\tilde y_{a_0}\cdots\tilde y_{a_{l-1}}))
 +r_{a_0-1}(\Delta^{\mathcal W}(\tilde y_{a_1}\cdots\tilde y_{a_{l}}))\Big)(\tilde y_{n}). 
\end{align*}
where the first equality follows from Lemma \ref{lem:form:of:der}, the second equality follows from the algebra 
morphism property of $\Delta^{\mathcal W}$ and \eqref{def:f(n)}, and the last equality follows from \eqref{12151309}. 
\end{proof}

\begin{prop}\label{prop:decomp}\label{prop:upper:comm:triangle}
The map $-\cdot\Delta^{\mathcal W} : \mathcal V_0\to\mathrm{Der}_{\Delta^{\mathcal W}}(\mathcal W,\mathcal W^{\otimes2})$ 
(see \eqref{act:on:Delta:1309}) is equal to the composed map 
$$
\mathcal V_0
\stackrel{\mathbf H}{\to}(\oplus_{k\geq0}\mathcal W^{\otimes2})^{\oplus 2}\oplus
(\oplus_{i\geq0}\mathcal W)
\stackrel{\mathbf h}{\to}
\mathrm{Map}(\mathbb Z_{>0},\mathcal W^{\otimes2})
\stackrel{\mathbf i}{\to}\mathrm{Der}_{\Delta^{\mathcal W}}(\mathcal W,\mathcal W^{\otimes2})
$$
\end{prop}

\begin{proof} 
Let $n>0$ and let $l\geq 0$, $a_0,\ldots,a_l>0$ with $a_0>1$ if $l=0$, and $v:=e_0^{a_0-1}e_1\cdots e_1e_0^{a_l-1}\in\mathcal V_0$. Then   
\begin{align}\label{interm:14391309}
& (v\cdot\Delta^{\mathcal W})(\tilde y_n)=(\mathrm{der}^{\mathcal W,(1)}_v\otimes\mathrm{id}+\mathrm{id}\otimes
\mathrm{der}^{\mathcal W,(1)}_v)\circ\Delta^{\mathcal W}(\tilde y_n)
-\Delta^{\mathcal W}\circ \mathrm{der}^{\mathcal W,(1)}_v(\tilde y_n)
\\&\nonumber 
=\sum_{j\geq0}m_{j}(\tilde y_{a_0}\cdots\tilde y_{a_{l-1}}\tilde y_{a_l-1-j}
-\tilde y_{a_0-1-j}\tilde y_{a_1}\cdots\tilde y_{a_{l}})(\tilde y_n)
\\&\nonumber  +\Big(\ell_{a_l-1}((\Delta^{\mathcal W}-\mathrm{id}\otimes1-1\otimes\mathrm{id})(\tilde y_{a_0}\cdots\tilde y_{a_{l-1}}))
-r_{a_0-1}((\Delta^{\mathcal W}-\mathrm{id}\otimes1-1\otimes\mathrm{id})(\tilde y_{a_1}\cdots\tilde y_{a_{l}}))\Big)(\tilde y_{n})
\end{align}
where the first equality follows from \eqref{act:on:Delta:1309} and the second equality from Lemmas \ref{lem:der:der:Delta:tilde:yn} 
and \ref{lem:Delta:der:a:tilde:yn} and from the linearity of $\ell_{a_l-1}$ and $r_{a_0-1}$. Using 
\eqref{eq:beg:1309}, \eqref{eq:end:1309}, \eqref{eq:mix:1309} and ${{L}}_k(v)=\delta_{k,a_l}\tilde y_{a_0}\cdots\tilde y_{a_{l-1}}$, ${{R}}_k(v)=\delta_{k,a_0}\tilde y_{a_1}\cdots\tilde y_{a_l}$, this implies  
\begin{equation}\label{07541309}
(v\cdot\Delta^{\mathcal{W}})(\tilde y_n)=\sum_{k\geq 0}\ell_k\circ{{L}}_{k+1}(v)(\tilde y_n)-\sum_{k\geq 0}r_k\circ {{R}}_{k+1}(v)(\tilde y_n)+\sum_{i\geq 0}m_i
\circ{{M}}_{i+1}(v)(\tilde y_n)\in\mathcal W^{\otimes2}. 
\end{equation}
For each $n>0$, both sides of this identity depend linearly on $v\in\mathcal V_0$, and since this identity is 
fulfilled for any $v$ in the family \eqref{basis:of:V}, which is a basis of $\mathcal V_0$, it holds for any 
$v\in\mathcal V_0$. Then for $v\in\mathcal V_0$, $v\cdot\Delta^{\mathcal{W}}$ and $\sum_{k\geq 0}\ell_k\circ{{L}}_{k+1}(v)
-\sum_{k\geq 0}r_k\circ {{R}}_{k+1}(v)+\sum_{i\geq 0}m_i\circ{{M}}_{i+1}(v)$ are elements of 
$\mathrm{Der}_{\Delta^{\mathcal{W}}}(\mathcal{W},\mathcal{W}^{\otimes2})$, and their images by the map 
$\mathbf i  : \mathrm{Der}_{\Delta^{\mathcal{W}}}(\mathcal{W},\mathcal{W}^{\otimes2})\to\mathrm{Map}(\mathbb Z_{>0},\mathcal W^{\otimes2})$
are the maps taking $n>0$ to respectively the left and right hand sides of \eqref{07541309}. By \eqref{07541309}, these images are equal, and 
the injectivity of $\mathbf i$ (see Lemma \ref{lem46:1401}) then implies the equality 
\begin{equation}\label{07471309}
\forall v\in\mathcal V,\quad v\cdot\Delta^{\mathcal{W}}=\sum_{k\geq 0}\ell_k\circ{{L}}_{k+1}(v)-\sum_{k\geq 0}r_k\circ {{R}}_{k+1}(v)+\sum_{i\geq 0}m_i
\circ{{M}}_{i+1}(v)
\in\mathrm{Der}_{\Delta^{\mathcal{W}}}(\mathcal{W},\mathcal{W}^{\otimes2}), 
\end{equation}
which by Definitions \ref{def:map:1409} and \ref{def:map':1409} implies the statement. 
\end{proof}

\section{Study of the constituents $\mathbf h$ and $\mathbf H$ of the map $\mathcal V_0\to\mathrm{Der}_{\Delta^{\mathcal W}}(\mathcal W,\mathcal W^{\otimes2})$}\label{sect:5:2501}

In this section, we study the maps $\mathbf h$ and $\mathbf H$. The first main result is the computation of $\mathrm{Ker}(\mathbf h)$. To obtain 
it, one first defines endomorphisms and degrees on the algebras $\mathcal W$ and $\mathcal W^{\otimes2}$ (\S\ref{441:1303}). These are used to 
construct sequences which are shown to be convergent in the discrete topology (i.e. eventually constant) in \S\ref{442:1303}. The opposite 
analogues of the results of \S\S\ref{441:1303} and \ref{442:1303} are obtained in \S\ref{subsub:opposite}, and  the results of 
\S\S\ref{441:1303}-\ref{subsub:opposite} are put together in \S\ref{subsub:comp:ker:h} to obtain the computation of $\mathrm{Ker}(\mathbf h)$ 
(Proposition \ref{descr:map'}). The second main result is a commutative square relating $\mathbf H$ with $\Delta^{\mathcal M}$ 
(\S\ref{subsub:comm:square}, Proposition \ref{comm:square}).  

\subsection{The endomorphisms $\partial_n$ and degrees $\mathrm{deg}$, $\mathrm{deg}^{(1)}$ and $\mathrm{deg}^{(2)}$} \label{441:1303}

\begin{lem}\label{id:partial:2908}
(a) For each $n>0$, there is a linear endomorphism $\partial_n$ of $\mathcal W$, uniquely determined by the conditions $\partial_n(a\tilde y_m)=\delta_{nm}a$ for any $a\in\mathcal W$ and $m>0$, and $\partial_n(1)=0$.   

(b) One has 
$$
\forall a,b\in\mathcal W,\quad \partial_n(ab)=a\partial_n(b)+\partial_n(a)\epsilon(b).
$$
 where $\epsilon:\mathcal W\to\mathbb Q$ is the projection of $\mathcal W$ on its degree 0 part.
\end{lem}

\begin{proof} (a) As the algebra $\mathcal W$ is freely generated by the family $(\tilde y_n)_{n>0}$, a basis is given by the set of all the words in these elements. 
For each $n>0$, the conditions uniquely determine the image of this basis, which determines $\partial_n$ uniquely. (b) can be checked for $a,b$ elements of this basis, which  by linearity implies its validity for any $a,b$. 
\end{proof}

\begin{lemdef}\label{lemdef:09100609}
(a) One has $\mathcal W=\mathbb Q\oplus (\oplus_{k>0}\mathcal W\tilde y_k)$.

(b) For $a\in\mathcal W$, one defines $\mathrm{deg}(a):=\min\{d\geq0|a\in\mathbb Q\oplus (\oplus_{k=1}^d\mathcal W\tilde y_l)\}\in\mathbb Z_{\geq 0}$.
\end{lemdef}

\begin{proof}
(a) follows from the fact that $\mathcal W$ is freely generated by the family $(\tilde y_n)_{n>0}$. It implies that the family 
$(\mathbb Q\oplus (\oplus_{k=1}^d\mathcal W\tilde y_k))_{d\geq0}$ is an increasing sequence of subsets of $\mathcal W$ whose union is 
$\mathcal W$, 
which justifies the definition (b).  
\end{proof}

\begin{lem}\label{lem:1725:2908}
For $a\in\mathcal W$,  one has $\partial_n(a)=0$ for any $n>\mathrm{deg}(a)$.
\end{lem}

\begin{proof} One has $a\in\mathbb Q\oplus (\oplus_{k=1}^{\mathrm{deg}(a)}\mathcal W\tilde y_k)$, and the restriction of $\partial_n$ to this space is $0$ for any $n>\mathrm{deg}(a)$. 
\end{proof}

\begin{lemdef}\label{lemdef:09092021}
(a) One has $\mathcal W^{\otimes2}=\mathbb Q\otimes\mathcal W\oplus(\oplus_{k>0}\mathcal W\tilde y_k\otimes\mathcal W)=\mathcal W\otimes\mathbb Q\oplus (\oplus_{k>0}\mathcal W\otimes\mathcal W\tilde y_k)$. 

(b) For $a\in\mathcal W^{\otimes2}$, one defines $\mathrm{deg}^{(1)}(a):=\min\{d\geq0|a\in
\mathbb Q\otimes\mathcal W\oplus(\oplus_{k=1}^d\mathcal W\tilde y_k\otimes\mathcal W)\}\in\mathbb Z_{\geq 0}$ and  $\mathrm{deg}^{(2)}(a):=\min\{d\geq0|a\in
\mathcal W\otimes\mathbb Q\oplus(\oplus_{k=1}^d\mathcal W\otimes\mathcal W\tilde y_k)\}\in\mathbb Z_{\geq 0}$.
\end{lemdef}

\begin{proof}
(a) follows from Lemma-Definition \ref{lemdef:09100609}, (a). It implies that 
$(\mathbb Q\otimes\mathcal W\oplus (\oplus_{k=1}^d\mathcal W\tilde y_k\otimes\mathcal W))_{d\geq0}$ and 
$(\mathcal W\otimes\mathbb Q\oplus (\oplus_{k=1}^d\mathcal W\otimes\mathcal W\tilde y_k))_{d\geq0}$ both are increasing 
sequences of subsets of $\mathcal W^{\otimes2}$ with union $\mathcal W^{\otimes2}$, 
which justifies the definition in (b).  
\end{proof}

\begin{lem}\label{17002908}
For any $a\in\mathcal W^{\otimes2}$, one has 
$$
(\partial_n\otimes\mathrm{id})(a)=0\quad \text{for}\quad n>\mathrm{deg}^{(1)}(a)\quad\text{and}\quad (\mathrm{id}\otimes\partial_n)(a)=0\quad\text{for}\quad n>\mathrm{deg}^{(2)}(a). 
$$
\end{lem}

\begin{proof}
One has $a\in(\mathbb Q\oplus(\oplus_{k=1}^{\mathrm{deg}^{(1)}(a)}\mathcal W\tilde y_k))\otimes\mathcal W$, and the restriction of $\partial_n\otimes\mathrm{id}$ to this space is $0$ for $n>\mathrm{deg}^{(1)}(a)$. 
Similarly, $a\in\mathcal W\otimes(\mathbb Q\oplus(\oplus_{k=1}^{\mathrm{deg}^{(2)}(a)}\mathcal W\tilde y_k))$, and the restriction of   $\mathrm{id}\otimes\partial_n$ to this space is $0$ for $n>\mathrm{deg}^{(2)}(a)$. 
\end{proof}

\subsection{Convergence results in the discrete topology}\label{442:1303}

Recall that any set $S$ can be equipped with its {\it discrete topology}. A sequence $(s_n)_{n>0}$ with values in $S$ is {\it convergent} 
in this topology iff there exists $s\in S$, such that $s_n=s$ for all but finitely many values of $n$. If $(s_n)_{n>0}$ is convergent, an element 
$s\in S$ with this property is necessarily unique and called the {\it limit} of $(s_n)_{n>0}$. 

\begin{lem}\label{lem:18192908}
If $(\underline a,\underline b,\underline z)\in(\oplus_{l\geq0}\mathcal W^{\otimes2})^{\oplus 2}\oplus(\oplus_{i\geq0}\mathcal W)$, then for any 
$k\geq0$, the sequence $\mathbb Z_{>0}\ni n\mapsto 
(\partial_n\otimes\partial_{n+k})(\mathbf h(\underline a,\underline b,\underline z)(2n))\in\mathcal{W}^{\otimes2}$ is convergent in the discrete 
topology of $\mathcal W^{\otimes2}$, with limit $a_k-\epsilon^{\otimes2}(b_k)1^{\otimes2}$, where $\underline a=(a_l)_{l\geq0}$, $\underline 
b=(b_l)_{l\geq0}$.  
\end{lem}

\begin{proof} Define $z_i\in\mathcal W$ for $i\geq 0$ by $\underline z=(z_i)_{i\geq0}\in\oplus_{i\geq0}\mathcal W$. 
Since the sequence $(b_{l})_{l\geq0}$ takes values zero for all but a finite number of indices, the same is true of the sequences $(\mathrm{deg}^{(1)}(b_{l}))_{l\geq0}$,  $(\mathrm{deg}^{(2)}(b_{l}))_{l\geq0}$, so these sequences are bounded. Since the sequence $(z_i)_{i\geq0}$ takes values zero for all but a finite number of indices, the same is true of the sequence $(\mathrm{deg}(z_i))_{i\geq0}$,  so this sequence is bounded. Set 
$$
N(\underline b,\underline z):=\mathrm{max}\Big\{\mathrm{max}\{\mathrm{deg}^{(1)}(b_{l})|l\geq0\},\mathrm{max}\{\mathrm{deg}^{(2)}(b_{l})|l\geq0\},\mathrm{max}\{\mathrm{deg}(z_i)|i\geq0\}\Big\}.
$$
Let $n>N(\underline b,\underline z)$ and let $k\geq0$. 

For any $l\geq0$, one has $(\partial_n\otimes\partial_{n+k})(a_{l}\Delta(2n+l))=\sum_{i=0}^{2n+l}a_{l}\delta_{i,n}\delta_{2n+l-i,n+k}$ therefore 
\begin{equation}\label{RH1}
\forall l\geq 0,\quad (\partial_n\otimes\partial_{n+k})(a_{l}\Delta(2n+l))=\delta_{kl}a_k. 
\end{equation} 

Let $l\geq0$. One has $n>\mathrm{deg}^{(1)}(b_{l})$ and $n+k>\mathrm{deg}^{(2)}(b_{l})$, which by Lemma \ref{17002908} implies
$(\partial_n\otimes\mathrm{id})(b_{l})=(\mathrm{id}\otimes\partial_{n+k})(b_{l})=0$ and therefore 
\begin{equation}\label{09092021}
(\partial_n\otimes\partial_{n+k})(b_{l})=(\epsilon\otimes\partial_{n+k})(b_{l})=(\partial_n\otimes\epsilon)(b_{l})=0. 
\end{equation}
Then    
\begin{align*}
& \nonumber (\partial_n\otimes\partial_{n+k})(\Delta(2n+l)b_{l})=\Delta(2n+l)(\partial_n\otimes\partial_{n+k})(b_{l})
+(\partial_n\otimes\mathrm{id})(\Delta(2n+l))(\epsilon\otimes\partial_{n+k})(b_{l})
\\&\nonumber +(\mathrm{id}\otimes\partial_{n+k})(\Delta(2n+l))(\partial_n\otimes\epsilon)(b_{l})
+(\partial_n\otimes\partial_{n+k})(\Delta(2n+l))\epsilon^{\otimes2}(b_{l})
\\&=(\partial_n\otimes\partial_{n+k})(\Delta(2n+l))\epsilon^{\otimes2}(b_{l})
=(\sum_{i=0}^{2n+l}\delta_{i,n}\delta_{2n+l-i,n+k})\epsilon^{\otimes2}(b_{l})1^{\otimes2}
=\delta_{kl}\epsilon^{\otimes2}(b_{l})1^{\otimes2}, 
\end{align*}
where the first equality follows from Lemma \ref{id:partial:2908} (b), the second equality from \eqref{09092021} and the third 
equality from $\partial_k(\tilde y_i)=\delta_{ki}$ and the definition of $\Delta(2n+l)$. Therefore
\begin{equation}\label{RH2}
\forall l\geq0, \quad
(\partial_n\otimes\partial_{n+k})(\Delta(2n+l)b_{l})=\delta_{kl}\epsilon^{\otimes2}(b_{l})1^{\otimes2}. 
\end{equation}

For any $i\geq0$, one has $\partial_n(z_i)=\partial_{n+k}(z_i)=0$ by Lemma \ref{lem:1725:2908}, therefore 
\begin{equation}\label{RH3}
\forall i\geq0,\quad (\partial_n\otimes\partial_{n+k})(z_i\otimes \tilde y_{2n+i}+\tilde y_{2n+i}\otimes z_i)=0. 
\end{equation}

Subtracting the sum for all the values of $l\geq0$ of \eqref{RH2} from the similar sum for \eqref{RH1}, adding up 
the sum for all the values of $i\geq0$ of \eqref{RH3}, and applying $\partial_n\otimes\partial_{n+k}$ to the expression \eqref{explicit:map'} of 
$\mathbf h(\underline a,\underline b,\underline z)(2n)$, one derives for any $k\geq0$: 
\begin{equation*}
\forall n>N(\underline b,\underline z),\quad (\partial_n\otimes\partial_{n+k})(\mathbf h(\underline a,\underline b,\underline z)(2n))=a_k-\epsilon^{\otimes2}(b_k)1^{\otimes2},  
\end{equation*}
which implies the result. \end{proof}

\begin{lem}\label{lem:1836:290821}
If $\underline z=(z_j)_{j\geq0}\in\oplus_{j\geq0}\mathcal W$, then for any $i\geq0$, the sequence $\mathbb Z_{>0}\ni n\mapsto 
((\epsilon\circ\partial_{n+i})\otimes\mathrm{id})(\mathbf h(0,0,\underline z)(n))\in\mathcal{W}$ is convergent in the discrete topology of 
$\mathcal W$, with limit $z_i$.  
\end{lem}

\begin{proof}
If $\underline z\in\oplus_{j\geq0}\mathcal W$, then the sequence $j\mapsto  \mathrm{deg}(z_j)$ is bounded. Set 
$N(\underline z):=\mathrm{max}\{\mathrm{deg}(z_j)|j\geq0\}$. Then if $n>N(\underline z)$ and $i\geq 0$, then $n+i>N(\underline z)\geq 
\mathrm{deg}(z_i)$, which by Lemma \ref{lem:1725:2908} implies $\partial_{n+i}(z_i)=0$. Therefore
\begin{equation}\label{vanishing}
\text{if } n>N(\underline z)\text{ and }i\geq0, \text{ then }\partial_{n+i}(z_i)=0.
\end{equation}
 Then 
if $n>N(\underline z)$ and $i\geq0$, 
$$
((\epsilon\circ\partial_{n+i})\otimes\mathrm{id})(\mathbf h(0,0,\underline z)(n))
=((\epsilon\circ\partial_{n+i})\otimes\mathrm{id})(z_i\otimes\tilde y_{n+i}+\tilde y_{n+i}\otimes z_i)=0+z_i=z_i,  
$$
where the first equality follows from \eqref{explicit:map'}, and the second equation from $\partial_{n+i}(\tilde y_{n+i})=1$ and \eqref{vanishing}. 
This implies the result. \end{proof}

\subsection{Opposite versions of the results of \S\S\ref{441:1303} and \ref{442:1303}}\label{subsub:opposite}

\begin{lem}\label{id:partial:0909}
(a) For each $n>0$, there is a linear endomorphism $\partial'_n$ of $\mathcal W$, uniquely determined by the conditions $\partial'_n(\tilde y_m a)=\delta_{nm}a$ for any $a\in\mathcal W$ and $m>0$, and $\partial'_n(1)=0$.   

(b) One has 
$\partial'_n(ab)=\partial'_n(a)b+\epsilon(a)\partial'_n(b)$ for any $a,b\in\mathcal W$.
\end{lem}

\begin{proof}
There is a unique automorphism $\mathrm{inv}$ of the vector space $\mathcal W$, which is the identity on $1$ and the $\tilde y_m$, 
$m\geq 1$ and satisfies $\mathrm{inv}(ab)=\mathrm{inv}(b)\mathrm{inv}(a)$. One checks that for any $n\geq1$, $\partial'_n=\mathrm{inv}
\circ\partial_n\circ\mathrm{inv}$. The results then follow from Lemma \ref{id:partial:2908}. 
\end{proof}

\begin{lemdef}\label{lemdef:09092021:bis}
(a) One has $\mathcal W=\mathbb Q\oplus (\oplus_{k>0}\tilde y_k\mathcal W)$.

(b) For $a\in\mathcal W$, one defines $\mathrm{deg}'(a):=\min\{d\geq0|a\in\mathbb Q\oplus (\oplus_{k=1}^d\tilde y_k\mathcal W)\}\in\mathbb Z_{\geq 0}$.
\end{lemdef}

\begin{proof}
(a) follows from Lemma-Definition \ref{lemdef:09100609} by applying $\mathrm{inv}$ (see proof of Lemma \ref{id:partial:0909}). 
Then $\mathrm{deg}'=\mathrm{deg}\circ\mathrm{inv}$. 
\end{proof}

\begin{lem}\label{lem:1725:0909}
For $a\in\mathcal W$, one has $\partial'_n(a)=0$ for any $n>\mathrm{deg}'(a)$.
\end{lem}

\begin{proof}
Follows from Lemma \ref{lem:1725:2908} and $\mathrm{deg}'=\mathrm{deg}\circ\mathrm{inv}$. 
\end{proof}

\begin{lemdef}
(a) One has $\mathcal W^{\otimes2}=\mathbb Q\otimes\mathcal W\oplus(\oplus_{k>0}\tilde y_k\mathcal W\otimes\mathcal W)=\mathcal W\otimes\mathbb Q\oplus (\oplus_{k>0}\mathcal W\otimes\tilde y_k\mathcal W)$. 

(b) For $a\in\mathcal W^{\otimes2}$, one defines $\mathrm{deg}^{\prime(1)}(a):=\min\{d\geq0|a\in
\mathbb Q\otimes\mathcal W\oplus(\oplus_{k=1}^d\tilde y_k\mathcal W\otimes\mathcal W)\}\in\mathbb Z_{\geq 0}$ and  $\mathrm{deg}^{\prime(2)}(a):=\min\{d\geq0|a\in
\mathcal W\otimes\mathbb Q\oplus(\oplus_{k=1}^d\mathcal W\otimes\tilde y_k\mathcal W)\}\in\mathbb Z_{\geq 0}$.
\end{lemdef}

\begin{proof}
(a) follows from Lemma-Definition \ref{lemdef:09092021} by applying $\mathrm{inv}^{\otimes 2}$. 
Then $\mathrm{deg}^{\prime (1)}=\mathrm{deg}^{(1)}\circ\mathrm{inv}^{\otimes2}$ 
and $\mathrm{deg}^{\prime (2)}=\mathrm{deg}^{(2)}\circ\mathrm{inv}^{\otimes2}$. 
\end{proof}

\begin{lem}\label{17002908:bis}
For any $a\in\mathcal W^{\otimes2}$, one has 
$$
(\partial'_n\otimes\mathrm{id})(a)=0\quad \mathrm{for}\quad n>\mathrm{deg}^{\prime(1)}(a)\quad\mathrm{for}\quad (\mathrm{id}\otimes\partial'_n)(a)=0\quad\mathrm{for}\quad n>\mathrm{deg}^{\prime(2)}(a). 
$$
\end{lem}

\begin{proof}
Follows Lemma \ref{17002908} by applied to $\mathrm{inv}^{\otimes2}(a)$. 
\end{proof}

\begin{lem}\label{lem:18192908bis}
If $(\underline a,\underline b,\underline z)\in(\oplus_{l\geq0}\mathcal W^{\otimes2})^{\oplus 2}\oplus(\oplus_{i\geq0}\mathcal W)$, then for any 
$k\geq0$, the sequence $\mathbb Z_{>0}\ni n\mapsto 
(\partial'_n\otimes\partial'_{n+k})(\mathbf h(\underline a,\underline b,\underline z)(2n))\in\mathcal{W}^{\otimes2}$ is convergent in the 
discrete topology of $\mathcal W^{\otimes2}$, with limit $-b_k+\epsilon^{\otimes2}(a_k)1^{\otimes2}$, where $\underline a=(a_l)_{l\geq0}$, 
$\underline b=(b_l)_{l\geq0}$.  
\end{lem}

\begin{proof}
Set $\tilde N(\underline a,\underline z):=\mathrm{max}\{\mathrm{max}\{\mathrm{deg}^{\prime(1)}(a_{l})|l\geq0\},
\mathrm{max}\{\mathrm{deg}^{\prime(2)}(a_{l})|l\geq0\},\mathrm{max}\{\mathrm{deg}'(z_i)|i\geq0\}\}$. Similarly to the proof of Lemma 
\ref{lem:18192908}, and using Lemmas \ref{lem:1725:0909} and \ref{lem:18192908bis}, 
one proves  
$$
\forall k\geq0,\quad\forall n>\tilde N(\underline a,\underline z),\quad (\partial'_n\otimes\partial'_{n+k})(\mathbf h(\underline a,\underline b,\underline z)(2n))=-b_k+\epsilon^{\otimes2}(a_k)1^{\otimes2} 
$$
which implies the result. 
\end{proof}

\subsection{Computation of $\mathrm{Ker}(\mathbf h)$}\label{subsub:comp:ker:h}

\begin{defn}
 $\mathbf{j} : \oplus_{k\geq0}\mathbb Q\to(\oplus_{k\geq0}\mathcal W^{\otimes2})^{\oplus 2}\oplus
(\oplus_{i\geq0}\mathcal W)$ is the linear map given by $\underline c\mapsto(\underline c1^{\otimes2},\underline c1^{\otimes2},0)$. 
\end{defn}

\begin{prop}\label{descr:map'}
The sequence $\oplus_{k\geq0}\mathbb Q\stackrel{\mathbf{j}}{\to}(\oplus_{k\geq0}\mathcal W^{\otimes2})^{\oplus 2}\oplus
(\oplus_{i\geq0}\mathcal W)\stackrel{\mathbf h}{\to}\mathrm{Map}(\mathbb Z_{>0},\mathcal W^{\otimes2})$ is exact. 
\end{prop}

\begin{proof} Using \eqref{explicit:map'}, one shows $\mathbf h(\underline c1^{\otimes2},\underline c1^{\otimes2},0)=0$ for any 
$\underline c\in\oplus_{l\geq 0}\mathbb Q$, which implies $\mathbf h\circ \mathbf{j}=0$. It follows that 
$\mathrm{Im}(\mathbf{j})\subset\mathrm{Ker}(\mathbf h)$. 
Let us prove the opposite inclusion. 
If $(\underline a,\underline b,\underline z)\in(\oplus_{l\geq0}\mathcal W^{\otimes2})^{\oplus 2}\oplus
(\oplus_{i\geq0}\mathcal W)$ belongs to $\mathrm{Ker}(\mathbf h)$, then $\mathbf h(\underline a,\underline b,\underline z)=0$, 
which implies that for any $k\geq 0$, the sequence $((\partial_n\otimes\partial_{n+k})\circ 
\mathbf h(\underline a,\underline b,\underline z)(2n))_{n>0}$ is zero. Lemma \ref{lem:18192908} and the uniqueness of the limit of 
a convergent series in the discrete topology then implies 
\begin{equation}\label{eq:a:b:1}
\forall k\geq0,\quad a_k=\epsilon^{\otimes2}(b_k)1^{\otimes2}.
\end{equation} Similarly, 
for any $k\geq 0$, the sequence $((\partial'_n\otimes\partial'_{n+k})\circ 
\mathbf h(\underline a,\underline b,\underline z)(2n))_{n>0}$ is zero. Lemma \ref{lem:18192908bis} and the same uniqueness 
principle then implies  
\begin{equation}\label{eq:a:b:2}
\forall k\geq0,\quad b_k=\epsilon^{\otimes2}(a_k)1^{\otimes2}.
\end{equation}
Equations \eqref{eq:a:b:1} and \eqref{eq:a:b:2} imply that for any $k\geq0$, 
$\epsilon^{\otimes2}(a_k)
=\epsilon^{\otimes2}(b_k)$, and if one sets $c_k:=\epsilon^{\otimes2}(a_k)
=\epsilon^{\otimes2}(b_k)\in\mathbb Q$, then $a_k=c_k1^{\otimes2}=b_k$. Set now $\underline c:=(c_k)_{k\geq0}\in\oplus_{k\geq0}\mathbb Q$, 
then $\underline a=\underline c1^{\otimes2}=\underline b$. Then $0=\mathbf h(\underline a,\underline b,\underline z)=\mathbf h(\underline c1^{\otimes2},\underline c1^{\otimes2},0)+\mathbf h(0,0,\underline z)=\mathbf h(0,0,\underline z)$, where the first (resp. second, third) equality follows from $(\underline a,\underline b,\underline z)\in\mathrm{Ker}(\mathbf h)$ (resp. $\underline a=\underline c1^{\otimes2}=\underline b$, $(\underline c1^{\otimes2},\underline c1^{\otimes2},0)\in\mathrm{Ker}(\mathbf h)$). 
Therefore $\mathbf h(0,0,\underline z)=0$. It follows that for any $i\geq0$, the sequence 
$(((\epsilon\circ\partial_{n+i})\otimes\mathrm{id})\circ \mathbf h(0,0,\underline z)(n))_{n>0}$ is zero. 
Then Lemma \ref{lem:1836:290821}, together with the uniqueness of a limit in the discrete topology 
implies $z_i=0$ for any $i\geq0$, hence $\underline z=0$. 
So $(\underline a,\underline b,\underline z)=(\underline c1^{\otimes2},\underline c1^{\otimes2},0)=\mathbf{j}(\underline c)\in\mathrm{Im}(\mathbf{j})$. 
\end{proof}

\subsection{A commutative square relating $\mathbf H$ and $\Delta^{\mathcal M}$ and a commutative triangle}\label{subsub:comm:square}

\begin{prop}\label{comm:square}
The following diagram is commutative  
$$
\xymatrix{
\mathcal V_0\ar^{\!\!\!\!\!\!\!\!\!\!\!\!\!\!\!\!\!\!\!\!\!\!
\!\!\!\!\!\!\!\!\!\!\!\!\!\!\!\!\!\!\!\!\!\!\!\!\!
\mathbf{H}}[r]\ar_{-\cdot 1_{\mathcal M}}[d]&
(\oplus_{k\geq0}\mathcal W^{\otimes2})\oplus
(\oplus_{k\geq0}\mathcal W^{\otimes2})\oplus
(\oplus_{i\geq0}\mathcal W)
\ar^{(-\cdot 1_{\mathcal M})^{\otimes 2}\circ (p_0\oplus 0\oplus 0)}[d]
\\ 
\mathcal M\ar_{\Delta^{\mathcal M}-\mathrm{id}\otimes1_{\mathcal M}-1_{\mathcal M}\otimes \mathrm{id}}[r]
& \mathcal M^{\otimes2}
}
$$
where $p_0 : \oplus_{k\geq0}\mathcal W^{\otimes2}\to\mathcal W^{\otimes2}$
is the projection on the summand $k=0$. 
\end{prop}

\begin{proof}
One checks that the restriction of ${{L}}_1:\mathcal V_0\to\mathcal W^{\otimes2}$ to $\mathcal Ve_0$ is zero, and that its 
restriction to $\mathcal Ve_1\subset\mathcal W$ coincides with $\Delta^{\mathcal{W}}-\mathrm{id}\otimes1-1\otimes\mathrm{id}$. 
This implies
\begin{equation}\label{beg:1:Delta:W}
    {{L}}_1=(\Delta^{\mathcal{W}}-\mathrm{id}\otimes1-1\otimes\mathrm{id})\circ\pi, 
\end{equation}
where $\pi : \mathcal V_0\to\mathcal W$ is the map whose restriction to $\mathcal Ve_0$ is the injection $\mathcal Ve_0\subset\mathcal W$
and with kernel $\mathcal Ve_0$. 
Let $v\in\mathcal V$. Then 
\begin{align*}
    & (-\cdot 1_{\mathcal M})^{\otimes2}\circ(p_0\oplus 0\oplus 0)(\mathbf{H}(v))={{L}}_1(v)\cdot 1_{\mathcal M}^{\otimes2} 
    =(\Delta^{\mathcal{W}}(\pi(v))-\pi(v)\otimes1-1\otimes\pi(v))\cdot 1_{\mathcal M}^{\otimes2}
    \\ & = \Delta^{\mathcal{M}}(\pi(v)\cdot 1_{\mathcal M})-\pi(v)\cdot 1_{\mathcal M}\otimes 1_{\mathcal M}-1_{\mathcal M}\otimes\pi(v)\cdot 1_{\mathcal M}
    =\Delta^{\mathcal{M}}(v\cdot 1_{\mathcal M})-v\cdot 1_{\mathcal M}\otimes 1_{\mathcal M}-1_{\mathcal M}\otimes v\cdot 1_{\mathcal M}
\end{align*}
where the first equality follows from $(p_0\oplus 0\oplus 0)(\mathbf{H}(v))={{L}}_1(v)$ (see Definition \ref{def:map:1409}), 
the second equality follows from \eqref{beg:1:Delta:W}, the third equality from \eqref{def:Delta:W}, and the last equality from 
$\pi(v)\cdot 1_{\mathcal M}=v\cdot 1_{\mathcal M}$ for any $v\in \mathcal V_0$. This implies the commutativity of the square. 
\end{proof}

\begin{prop}\label{easy:comm:triangle}
The following triangle is commutative 
$$
\xymatrix{\oplus_{k\geq0}\mathbb Q\ar^{p_0\cdot 1_{\mathcal M}^{\otimes2}}[rr]\ar_{\mathbf{j}}[rd]&&\mathcal M^{\otimes2}
\\&(\oplus_{k\geq0}\mathcal W^{\otimes2})^{\oplus2}\oplus(\oplus_{i\geq0}\mathcal W)\ar_{(-\cdot 1_{\mathcal M})^{\otimes2}\circ(p_0\oplus 0\oplus 0)}[ur]&}
$$
where $p_0$ is as in Proposition \ref{comm:square} and $\tilde p_0 : \oplus_{k\geq0}\mathbb Q\to\mathbb Q$ is the projection on the 
component $k=0$. 
\end{prop}

\begin{proof} Let $\underline c=(c_k)_{k\geq0}\in\oplus_{k\geq0}\mathbb Q$. Then 
\begin{align*}
&(-\cdot 1_{\mathcal M})^{\otimes2}\circ(p_0\oplus 0\oplus 0)\circ \mathbf{j}(\underline c)=
(-\cdot 1_{\mathcal M})^{\otimes2}\circ(p_0\oplus 0\oplus 0)(\underline c1^{\otimes2},\underline c1^{\otimes2},0)
\\&=(-\cdot 1_{\mathcal M})^{\otimes2}(c_0 1^{\otimes2})=c_0 1_{\mathcal M}^{\otimes2}=
\tilde p_0(\underline c)1_{\mathcal M}^{\otimes2}. 
\end{align*}
where the first equality follows from 
$\mathbf{j}(\underline c)=(\underline c1^{\otimes2},\underline c1^{\otimes2},0)$, 
the second equality follows from 
$(p_0\oplus 0\oplus 0)(\underline c1^{\otimes2},\underline c1^{\otimes2},0)=c_0 1^{\otimes2}$, 
and the last equality follows from $\tilde p_0(\underline c)=c_0$. 
\end{proof}

\section{Proof of equality results}

In \S\ref{subsection61}, we put together the results of \S\S\ref{sect:4:2501} and \ref{sect:5:2501} to prove the equality between
the Lie algebra $\mathfrak{stab}(\hat\Delta^{\mathcal W,\DR})$ and $\mathfrak{stab}(\hat\Delta^{\mathcal M,\DR})$ 
(Theorem \ref{thm:main}). In \S\ref{subsection:62:2501}, we derive from this the equality of the stabilizer bitorsors 
$\mathsf{Stab}(\hat\Delta^{\mathcal W,\DR/\B})$ 
and $\mathsf{Stab}(\hat\Delta^{\mathcal M,\DR/\B})$ (Theorem \ref{thm:24122021}). 

\subsection{Proof of the equality of the stabilizer Lie algebras $\mathfrak{stab}(\hat\Delta^{\mathcal{W,\DR}})$ and 
$\mathfrak{stab}(\hat\Delta^{\mathcal{M,\DR}})$}\label{subsection61}

The subspace $\mathcal P(\mathcal M)\subset\mathcal M$ is graded. 
 It follows that the same holds for the subspace $(-\cdot 1_{\mathcal M})^{-1}(\mathcal P(\mathcal M))\subset\mathcal V_0$.

\begin{prop}\label{thm:aux}
One has $\mathfrak{stab}_{\mathcal V_0}(\Delta^{\mathcal W})\subset
(-\cdot 1_{\mathcal M})^{-1}(\mathcal P(\mathcal M))$ (inclusion of graded subspaces of $\mathcal V_0$). 
\end{prop}

\begin{proof}
Both sides of this inclusion are graded, so it suffices to show the inclusion of homogeneous components for each degree. Let $a\in\mathfrak{stab}_{\mathcal V_0}(\Delta^{\mathcal W})\subset\mathcal V_0$ be homogeneous. 
Consider the diagram 
$$
\xymatrix{\mathcal M\ar_{\Delta^{\mathcal M}-\mathrm{id}\otimes1_{\mathcal M}-1_{\mathcal M}\otimes\mathrm{id}}[d]&&\mathcal V_0
\ar_{-\cdot1_{\mathcal M}}[ll]\ar^{-\cdot\Delta^{\mathcal W}}[r]\ar_{\mathbf{H}}[d]&
\mathrm{Der}_{\Delta^{\mathcal W}}(\mathcal W,\mathcal W^{\otimes2})\\
\mathcal M^{\otimes2}&&
(\oplus_{k\geq0}\mathcal W^{\otimes2})^{\oplus 2}\oplus
(\oplus_{i\geq0}\mathcal W)\ar_{\mathbf h}[r]\ar_{\!\!\!\!\!\!
\!\!\!\!\!\!\!\!\!\!\!\!\!\!\!\!\!\!\!\!
(-\cdot 1_{\mathcal M})^{\otimes2}\circ(p_0\oplus 0\oplus 0)}[ll]&\mathrm{Map}(\mathbb Z_{>0},\mathcal W^{\otimes 2})
\ar^\simeq_{\mathbf i}[u]\\
\oplus_{k\geq0}\mathbb Q\ar_{\mathbf{j}}[urr]\ar^{\tilde p_0\cdot 1_{\mathcal M}^{\otimes2}}[u]&&&}
$$
where the left square commutes by Proposition \ref{comm:square}, the right square commutes by Proposition \ref{prop:upper:comm:triangle}, 
the triangle commutes by Proposition \ref{easy:comm:triangle}, and the lower sequence of maps $(\mathbf h,\mathbf j)$ is exact by Proposition \ref{descr:map'}.  

Since $\mathfrak{stab}_{\mathcal V_0}(\Delta^{\mathcal W})=\mathrm{Ker}(-\cdot\Delta^{\mathcal W})$ (see Lemma \ref{lem:16393110} (a)) 
and by the commutativity of the right square, $\mathbf i\circ\mathbf h\circ \mathbf{H}(a)=0$, which by the injectivity of $\mathbf i$
(see Lemma \ref{lem46:1401}) implies that $\mathbf{H}(a)$ belongs to $\mathrm{Ker}(\mathbf h)$. The exactness of the lower sequence of maps 
then implies the existence of $\underline c\in\oplus_{l\geq0}\mathbb Q$, such that 
\begin{equation}\label{12191009}
\mathbf{H}(a)=\mathbf{j}(\underline c).    
\end{equation} 
Then 
$$
(\Delta^{\mathcal M}-\mathrm{id}\otimes1_{\mathcal M}-1_{\mathcal M}\otimes\mathrm{id})(a\cdot 1_{\mathcal M})=(-\cdot 1_{\mathcal M})^{\otimes2}\circ(p_0\oplus0\oplus0)\circ \mathbf{H}(a)
=(-\cdot 1_{\mathcal M})^{\otimes2}\circ(p_0\oplus0\oplus0)(\mathbf{j}(\underline c))
=\tilde p_0(\underline c)1_{\mathcal M}^{\otimes2}, 
$$
where the first equality follows from the commutativity of the left square, the second equality follows from \eqref{12191009}, 
and the third equality follows from the commutativity of the triangle. 
It follows that 
$(\Delta^{\mathcal M}-\mathrm{id}\otimes1_{\mathcal M}-1_{\mathcal M}\otimes\mathrm{id})(a\cdot1)\in\mathbb Q1_{\mathcal M}^{\otimes2}$. 
The map $a\mapsto(\Delta^{\mathcal M}-\mathrm{id}\otimes1_{\mathcal M}-1_{\mathcal M}\otimes\mathrm{id})(a\cdot1)$ is graded and
$a$ has positive degree, which implies that 
$(\Delta^{\mathcal M}-\mathrm{id}\otimes1_{\mathcal M}-1_{\mathcal M}\otimes\mathrm{id})(a\cdot1_{\mathcal M})=0$, which by \eqref{identif:PM:10112021} implies $a\cdot 1_{\mathcal M}\in\mathcal P(\mathcal M)$.  
\end{proof}

\begin{lem}\label{comp:der:2211}
The derivation $^\Gamma\mathrm{der}^{\mathcal W,(1)}_{[e_0,e_1]}$ of $\mathcal W$ is such that for any $n\geq1$, 
$$\tilde y_n\mapsto 
(\tilde y_2+{1\over 2}\tilde y_1^2)\tilde y_n+\tilde y_n(\tilde y_2-{1\over 2}\tilde y_1^2)
-\tilde y_1\tilde y_{n+1}-\tilde y_{n+1}\tilde y_1. 
$$ 
\end{lem}

\begin{proof} One computes $\theta([e_0,e_1])=[e_0,e_1]+(1/2)([e_0,e_1]|e_0e_1)e_1^2=[e_0,e_1]+(1/2)e_1^2$, therefore 
$^\Gamma\mathrm{der}^{\mathcal V,(1)}_{[e_0,e_1]}=
\mathrm{der}^{\mathcal V,(1)}_{\theta([e_0,e_1])}$
is given by $e_0\mapsto[[e_0,e_1]+e_1^2/2,e_0]$ and $e_1\mapsto0$. Then for $n\geq1$, 
\begin{align*}
    &    ^\Gamma\mathrm{der}^{\mathcal W,(1)}_{[e_0,e_1]}(\tilde y_n)=
\ ^\Gamma\!\mathrm{der}^{\mathcal V,(1)}_{[e_0,e_1]}(e_0^{n-1}e_1)
=[[e_0,e_1]+e_1^2/2,e_0^{n-1}]e_1
\\&=
(\tilde y_2+{1\over 2}\tilde y_1^2)\tilde y_n+\tilde y_n(\tilde y_2-{1\over 2}\tilde y_1^2)
-\tilde y_1\tilde y_{n+1}-\tilde y_{n+1}\tilde y_1. 
\end{align*}
\end{proof}

\begin{lem}\label{lem:046:2211}
The derivation $\theta([e_0,e_1])\cdot\Delta^{\mathcal W}=(^\Gamma\mathrm{der}^{\mathcal W,(1)}_{[e_0,e_1]}\otimes\mathrm{id}+
\mathrm{id}\otimes\ ^\Gamma\!\mathrm{der}^{\mathcal W,(1)}_{[e_0,e_1]})\circ\Delta^{\mathcal W}
-\Delta^{\mathcal W}\circ\ ^\Gamma\!\mathrm{der}^{\mathcal W,(1)}_{[e_0,e_1]}$ of $\mathcal W$ is the element of 
$\mathrm{Der}_{\Delta^{\mathcal W}}(\mathcal W,\mathcal W^{\otimes2})$ such that for any $n\geq1$, 
$$
\tilde y_n\mapsto 
2(\tilde y_1\otimes\tilde y_{n+1}+\tilde y_{n+1}\otimes\tilde y_1)
-2((\tilde y_2+\tilde y_1^2)\otimes\tilde y_n+\tilde y_n\otimes(\tilde y_2+\tilde y_1^2))
+2(\tilde y_n\otimes1+1\otimes\tilde y_n-\sum_{k=1}^{n-1}\tilde y_k\otimes\tilde y_{n-k})(\tilde y_1\otimes\tilde y_1)
$$ 
\end{lem}

\begin{proof}
Let  $\sigma$ be the permutation of tensor factor of $\mathcal W^{\otimes2}$. Then  
using Lemma \ref{comp:der:2211}, one computes, for $n\geq1$
\begin{align*}
    & (^\Gamma\mathrm{der}^{\mathcal W,(1)}_{[e_0,e_1]}\otimes\mathrm{id}+
\mathrm{id}\otimes ^\Gamma\mathrm{der}^{\mathcal W,(1)}_{[e_0,e_1]})\circ\Delta^{\mathcal W}(\tilde y_n)
\\&
=(\mathrm{id}+\sigma)\Big(
^\Gamma\mathrm{der}^{\mathcal W,(1)}_{[e_0,e_1]}(\tilde y_n)\otimes1
-\sum_{k=1}^{n-1}
\ ^\Gamma\mathrm{der}^{\mathcal W,(1)}_{[e_0,e_1]}(\tilde y_k)\otimes\tilde y_{n-k}
\Big)
\end{align*}
where the equality follows from \eqref{def:Delta:W} and from the $\sigma$-invariance of $\Delta^{\mathcal W}(\tilde y_n)$, 
and
\begin{align*}
    & \Delta^{\mathcal W}\circ\ ^\Gamma\!\mathrm{der}^{\mathcal W,(1)}_{[e_0,e_1]}(\tilde y_n)
    =\Delta^{\mathcal W}((\tilde y_2+{1\over 2}\tilde y_1^2)\tilde y_n+\tilde y_n(\tilde y_2-{1\over 2}\tilde y_1^2)
    -\tilde y_1\tilde y_{n+1}-\tilde y_{n+1}\tilde y_1)
\\ & =
(\mathrm{id}+\sigma)\Big(((\tilde y_2+{1\over 2}\tilde y_1^2)\otimes1)(\tilde y_n\otimes1+1\otimes \tilde y_n
-\sum_{k=1}^{n-1}\tilde y_k\otimes\tilde y_{n-k})
\\ & +(\tilde y_n\otimes1+1\otimes \tilde y_n
-\sum_{k=1}^{n-1}\tilde y_k\otimes\tilde y_{n-k})((\tilde y_2-{1\over 2}\tilde y_1^2)\otimes1-\tilde y_1\otimes\tilde y_1)
\\ & -(\tilde y_1\otimes 1)(\tilde y_{n+1}\otimes1+1\otimes \tilde y_{n+1}
-\tilde y_1\otimes\tilde y_n-\sum_{k=1}^{n-1}\tilde y_{k+1}\otimes\tilde y_{n-k})
\\ & -(\tilde y_{n+1}\otimes1+1\otimes \tilde y_{n+1}
-\tilde y_1\otimes\tilde y_n-\sum_{k=1}^{n-1}\tilde y_{k+1}\otimes\tilde y_{n-k})(\tilde y_1\otimes 1)
\Big)
\end{align*}
where the first equality follows from Lemma \ref{comp:der:2211} and the second equality from the identities 
$(\mathrm{id}+\sigma)(a)b=(\mathrm{id}+\sigma)(ab)$ 
and $b(\mathrm{id}+\sigma)(a)=(\mathrm{id}+\sigma)(ba)$ for $a,b\in\mathcal W^{\otimes2}$ such that $b$ is $\sigma$-invariant, 
the $\sigma$-invariance and the explicit expressions of $\Delta^{\mathcal W}(\tilde y_n)$ and $\Delta^{\mathcal W}(\tilde y_{n+1})$, 
and the equalities 
$\Delta^{\mathcal W}(\tilde y_2+{1\over 2}\tilde y_1^2)=(\mathrm{id}+\sigma)((\tilde y_2+{1\over 2}\tilde y_1^2)\otimes1)$,
$\Delta^{\mathcal W}(\tilde y_2-{1\over 2}\tilde y_1^2)=(\mathrm{id}+\sigma)((\tilde y_2-{1\over 2}\tilde y_1^2)\otimes1-\tilde y_1\otimes\tilde 
y_1)$ and $\Delta^{\mathcal W}(\tilde y_1)=(\mathrm{id}+\sigma)(\tilde y_1\otimes1)$. 
The result follows from the computation of the difference of the two expressions. 
\end{proof}

\begin{rem}
It follows that $\theta([e_0,e_1])\cdot\Delta^{\mathcal W}=\mathbf i\circ\mathbf h(\underline a^0,\underline b^0,\underline z^0)$ where 
$\underline a^0=0$, $\underline b^0=(b^0_i)_{i\geq 0}$ where $b^0_0=2\tilde y_1\otimes\tilde y_1$ and $b^0_i=0$ for $i>0$,  
$\underline z^0=(z^0_i)_{i\geq0}$ where $z^0_0=-2(\tilde y_2+\tilde y_1^2)$,  $z^0_1=2\tilde y_1$, and $z^0_i=0$ for $i>1$. 
One can check that $\mathbf H([e_0,e_1])=(\underline a^0,\underline b^0,\underline z^0)$. 
\hfill\qed\medskip
\end{rem}

Recall that $\mathfrak{stab}_{\mathfrak{lie}(e_0,e_1)}(\Delta^{\mathcal W})$ is a graded Lie subalgebra of 
$(\mathfrak{lie}(e_0,e_1),\langle,\rangle)$. 

\begin{lem}\label{lemma:van:deg:2:component}
The degree $2$ component of $\mathfrak{stab}_{\mathfrak{lie}(e_0,e_1)}(\Delta^{\mathcal W})$ is zero. 
\end{lem}

\begin{proof} Since the degree 2 part of $\mathfrak{lie}(e_0,e_1)$ is one-dimensional, spanned by $[e_0,e_1]$, the statement 
is equivalent to $[e_0,e_1]\notin\mathfrak{stab}_{\mathfrak{lie}(e_0,e_1)}(\Delta^{\mathcal W})$, which follows from 
Lemma \ref{lem:046:2211}.
\end{proof}

\begin{thm}\label{thm:main}
The Lie subalgebras $\mathfrak{stab}(\hat\Delta^{\mathcal{W,\DR}})$ and $\mathfrak{stab}(\hat\Delta^{\mathcal{M,\DR}})$
of $\mathfrak g^{\DR}$ (see \cite{EF2}, \S3.5) are equal. 
\end{thm}

\begin{proof}
The inclusion $\mathfrak{stab}(\hat\Delta^{\mathcal{M}})\subset\mathfrak{stab}(\hat\Delta^{\mathcal{W}})$ follows from 
\cite{EF2}, Corollary 3.14 (c). By Propositions \ref{lem:08:1711} and \ref{lem:015:1711}, it implies the inclusion 
\begin{equation}\label{incl:LAs:1711}
\mathfrak{stab}_{\mathfrak{lie}(e_0,e_1)}(\Delta^{\mathcal{M}})\subset
\mathfrak{stab}_{\mathfrak{lie}(e_0,e_1)}(\Delta^{\mathcal{W}})    
\end{equation}
of graded Lie subalgebras of $\mathfrak{lie}(e_0,e_1)$. 

On the other hand, one has 
\begin{align*}
& \mathfrak{stab}_{\mathfrak{lie}(e_0,e_1)}(\Delta^{\mathcal{W}})\cap 
\mathfrak{lie}(e_0,e_1)[\check 2]
=\theta^{-1}(\mathfrak{stab}_{\mathcal V_0}(\Delta^{\mathcal{W}}))
\cap \mathfrak{lie}(e_0,e_1)[\check 2]
\\ & \subset\theta^{-1}(\mathbb Q1\oplus\mathcal P(\mathcal M))
\cap \mathfrak{lie}(e_0,e_1)[\check 2]
=\mathfrak{stab}_{\mathfrak{lie}(e_0,e_1)}(\Delta^{\mathcal{M}}). 
\end{align*}
where the first equality follows from Lemma \ref{lem:16393110} (c), the inclusion follows from 
Proposition \ref{thm:aux}, and the second equality follows from \eqref{+1459}. Together with \eqref{incl:LAs:1711}, 
this implies 
\begin{equation}\label{eq:neq:2:1711}
    \forall d\neq 2,\quad \mathfrak{stab}_{\mathfrak{lie}(e_0,e_1)}(\Delta^{\mathcal{W}})[d]
=\mathfrak{stab}_{\mathfrak{lie}(e_0,e_1)}(\Delta^{\mathcal{M}})[d]. 
\end{equation}
Then $\mathfrak{stab}_{\mathfrak{lie}(e_0,e_1)}(\Delta^{\mathcal{W}})[2]=0=\mathfrak{stab}_{\mathfrak{lie}(e_0,e_1)}(\Delta^{\mathcal{M}})[2]$, 
where the first equality follows from Lemma \ref{lemma:van:deg:2:component} and the second equality follows from \eqref{+1459}. Together with 
\eqref{eq:neq:2:1711}, this implies the equality of the graded components of the Lie algebras 
$\mathfrak{stab}_{\mathfrak{lie}(e_0,e_1)}(\Delta^{\mathcal{W}})$ and $\mathfrak{stab}_{\mathfrak{lie}(e_0,e_1)}(\Delta^{\mathcal{M}})$ 
in any degree, hence the equality of these Lie algebras. The announced equality then follows from 
Propositions \ref{lem:08:1711} and \ref{lem:015:1711}. 
\end{proof}

\subsection{Proof of the equality of the stabilizer bitorsors $\mathsf{Stab}(\hat\Delta^{\mathcal M,\DR/\B})(\mathbf k)$ and 
$\mathsf{Stab}(\hat\Delta^{\mathcal W,\DR/\B})(\mathbf k)$}\label{subsection:62:2501}

\subsubsection{Review of known results}

In \cite{EF2}, Definitions 2.16 and 2.20 and \cite{EF3}, Definitions 3.5 and 3.8, we introduced $\mathbb Q$-group schemes 
$\mathsf{Stab}(\hat\Delta^{?,??})$, where $?$ (resp. $??$) is one of the symbols $\mathcal M,\mathcal W$ (resp. DR,B); these are extensions
of $\mathbb G_m$ by prounipotent group schemes and give rise to functors $\mathbf k\mapsto\mathsf{Stab}(\hat\Delta^{?,??})(\mathbf k)$ from the 
category of $\mathbb Q$-algebras to that of groups; the Lie algebras of $\mathsf{Stab}(\hat\Delta^{?,\DR})$ are 
$\mathfrak{stab}(\hat\Delta^{?,\DR})$ (see \cite{EF2}, \S3.5), and the Lie algebras $\mathfrak{stab}(\hat\Delta^{?,\B})$ of 
$\mathsf{Stab}(\hat\Delta^{?,\B})$ were defined in \cite{EF3}, \S3.8.  

In \cite{EF2}, Definitions 2.17 and 2.21, we introduced the $\mathbb Q$-schemes $\mathsf{Stab}(\hat\Delta^{?,\DR/\B})$, where $?$ is one 
of the symbols $\mathcal M,\mathcal W$, giving rise to functors $\mathbf k\mapsto\mathsf{Stab}(\hat\Delta^{?,\DR/\B})(\mathbf k)$
from the category of $\mathbb Q$-algebras to that of sets. For any $\mathbf k$, the set $\mathsf{Stab}(\hat\Delta^{?,\DR/\B})(\mathbf k)$
is equipped with commuting left and right actions of the groups $\mathsf{Stab}(\hat\Delta^{?,\DR})(\mathbf k)$ and 
$\mathsf{Stab}(\hat\Delta^{?,\B})(\mathbf k)$. By \cite{EF2}, Theorem 3.1 (b) and (c) and \cite{EF3}, Lemmas 3.6 and 3.9, each of 
these actions is free and transitive, so that $\mathsf{Stab}(\hat\Delta^{?,\DR/\B})(\mathbf k)$ is a {\it bitorsor} (see 
\cite{EF3}, Definition 1.1 (a), and \cite{Gi}, Chap. III, Definition 1.5.3).  

In \cite{EF2}, Theorem 3.1 (a) and \cite{EF3}, Theorem 3.14 (b) we constructed for any $\mathbf k$, group inclusions 
$\mathsf{Stab}(\hat\Delta^{\mathcal M,\DR})(\mathbf k)\subset\mathsf{Stab}(\hat\Delta^{\mathcal W,\DR})(\mathbf k)$ and 
$\mathsf{Stab}(\hat\Delta^{\mathcal M,\B})(\mathbf k)\subset
\mathsf{Stab}(\hat\Delta^{\mathcal W,\B})(\mathbf k)$ and set inclusions $\mathsf{Stab}(\hat\Delta^{\mathcal M,\DR/\B})
(\mathbf k)\subset\mathsf{Stab}(\hat\Delta^{\mathcal W,\DR/\B})(\mathbf k)$, which are compatible with the actions.  

\subsubsection{Equality results}

\begin{thm}\label{thm:24122021}
Let $\mathbf k$ be a $\mathbb Q$-algebra. Then the inclusion of the bitorsors attached to 
$\mathsf{Stab}(\hat\Delta^{\mathcal M,\DR/\B})(\mathbf k)$ and $\mathsf{Stab}(\hat\Delta^{\mathcal W,\DR/\B})(\mathbf k)$ is an equality, i.e.: 

(a) the group inclusion $\mathsf{Stab}(\hat\Delta^{\mathcal M,\DR})(\mathbf k)\subset\mathsf{Stab}(\hat\Delta^{\mathcal W,\DR})(\mathbf k)$ is an 
equality; 

(b) the inclusion $\mathsf{Stab}(\hat\Delta^{\mathcal M,\DR/\B})(\mathbf k)\subset\mathsf{Stab}(\hat\Delta^{\mathcal W,\DR/\B})(\mathbf k)$ is an 
equality; 

(c) the group inclusion $\mathsf{Stab}(\hat\Delta^{\mathcal M,\B})(\mathbf k)\subset\mathsf{Stab}(\hat\Delta^{\mathcal W,\B})(\mathbf k)$ is an 
equality. 
\end{thm}

\begin{proof} 
(a) The sets $\mathsf{Stab}(\hat\Delta^{?,\DR})(\mathbf k)$ for $?\in\{\mathcal M,\mathcal W\}$ are 
subgroups of the group $\mathsf G^\DR(\mathbf k)$, which is $\mathbf k^\times\times\mathcal G(\hat{\mathcal V})$ 
equipped with the structure of semidirect product of the group $(\mathcal G(\hat{\mathcal V}),\circledast)$
by the action of $\mathbf k^\times$ (see \cite{EF2}, 
Definitions 2.16 and 2.20 and §1.6.3, where $\hat{\mathcal V}$ is denoted $\hat{\mathcal V}^\DR$). By \cite{EF2}, Lemmas 2.18 and 2.22, 
$\mathsf{Stab}(\hat\Delta^{?,\DR})(\mathbf k)$ contain $\mathbf k^\times$, which implies the equalities 
$\mathsf{Stab}(\hat\Delta^{?,\DR})(\mathbf k)=\mathbf k^\times\times\mathsf{Stab}_1(\hat\Delta^{?,\DR})(\mathbf k)$ 
of subsets of $\mathbf k^\times\times\mathcal G(\hat{\mathcal V})$, where 
$\mathsf{Stab}_1(\hat\Delta^{?,\DR})(\mathbf k)$ are the images of the intersections of 
$\mathsf{Stab}(\hat\Delta^{?,\DR})(\mathbf k)$ with $\{1\}\times\mathcal G(\hat{\mathcal V})$ by the canonical isomorphism 
$\{1\}\times\mathcal G(\hat{\mathcal V})\to \mathcal G(\hat{\mathcal V})$ (see \cite{EF2}, Remark 2.23). 
The statement is therefore equivalent to the equality of the subgroups $\mathsf{Stab}_1(\hat\Delta^{\mathcal M,\DR})(\mathbf k)$
and $\mathsf{Stab}_1(\hat\Delta^{\mathcal W,\DR})(\mathbf k)$ of $(\mathcal G(\hat{\mathcal V}),\circledast)$. 

Equip the completion $\mathfrak{lie}(e_0,e_1)^\wedge\hat\otimes\mathbf k$ with the product $\mathrm{cbh}_{\langle,\rangle}(\cdot,\cdot)$ taking 
$(x,y)$ to the image by the Lie algebra morphism from the topologically free Lie algebra with generators $a,b$ 
to $(\mathfrak{lie}(e_0,e_1)^\wedge\hat\otimes\mathbf k,\langle,\rangle)$ given by $a\mapsto x$, $b\mapsto y$ 
of the element $\mathrm{log}(e^ae^b)$. Then $(\mathfrak{lie}(e_0,e_1)^\wedge\hat\otimes\mathbf k,\mathrm{cbh}_{\langle,\rangle}(\cdot,\cdot))$ 
is a group. The map 
$\mathrm{exp}_\circledast : \mathfrak{lie}(e_0,e_1)^\wedge\hat\otimes\mathbf k\to\mathcal G(\hat{\mathcal V})$ defined in \cite{R}, 
(3.1.10.1) sets up a group isomorphism $(\mathfrak{lie}(e_0,e_1)^\wedge\hat\otimes\mathbf k,\mathrm{cbh}_{\langle,\rangle}(\cdot,\cdot))
\simeq(\mathcal G(\hat{\mathcal V}),\circledast)$ (\cite{R}, Corollary 3.1.10, see also \cite{EF2}, \S3.5.2). It follows from 
\cite{EF2}, Lemma 3.14 (based on \cite{EF0}, Lemma 5.1) that $\mathrm{exp}_\circledast$ restricts to bijections 
$\mathfrak{stab}_0(\hat\Delta^{\mathcal M,\DR})\hat\otimes \mathbf k\to\mathsf{Stab}_1(\hat\Delta^{\mathcal M,\DR})(\mathbf k)$
and $\mathfrak{stab}_0(\hat\Delta^{\mathcal W,\DR})\hat\otimes \mathbf k\to\mathsf{Stab}_1(\hat\Delta^{\mathcal W,\DR})(\mathbf k)$, the
index 0 meaning the intersection of a Lie subalgebra of $\mathfrak g^\DR$ with $\mathfrak{lie}(e_0,e_1)^\wedge\hat\otimes\mathbf k$.
As Theorem \ref{thm:main} implies the equality $\mathfrak{stab}_0(\hat\Delta^{\mathcal M,\DR})\hat\otimes 
\mathbf k=\mathfrak{stab}_0(\hat\Delta^{\mathcal W,\DR})\hat\otimes \mathbf k$, one obtains the equality 
$\mathsf{Stab}_1(\hat\Delta^{\mathcal M,\DR})(\mathbf k)=\mathsf{Stab}_1(\hat\Delta^{\mathcal W,\DR})(\mathbf k)$.  

(b) follows from the fact that $\mathsf{Stab}(\hat\Delta^{\mathcal M,\DR/\B})(\mathbf k)$ and $\mathsf{Stab}(\hat\Delta^{\mathcal W,\DR/\B})
(\mathbf k)$ are subtorsors of $\mathsf{G}^{\DR,\B}(\mathbf k)$ with respective groups $\mathsf{Stab}(\hat\Delta^{\mathcal M,\DR})(\mathbf k)$ 
and $\mathsf{Stab}(\hat\Delta^{\mathcal W,\DR})(\mathbf k)$ (see \cite{EF2}, Theorem 3.1, (b) and (c)), from the torsor inclusion $\mathsf{Stab}
(\hat\Delta^{\mathcal M,\DR/\B})(\mathbf k)\subset\mathsf{Stab}(\hat\Delta^{\mathcal W,\DR/\B})(\mathbf k)$ 
(see \cite{EF2}, Theorem 3.1, (c)), from (a) and from \cite{EF2}, Lemma 2.7, (b). 

(c): (a) and (b) imply that the subtorsors $\mathsf{Stab}(\hat\Delta^{\mathcal M,\DR/\B})(\mathbf k)$ and 
$\mathsf{Stab}(\hat\Delta^{\mathcal W,\DR/\B})(\mathbf k)$ of $\mathsf{G}^{\DR,\B}(\mathbf k)$ are equal. If $_GX_H$ is 
a bitorsor and if $_{G'}X'_{H'}$ and $_{G''}X''_{H''}$ are subbitorsors such that $_{G'}X'$ and $_{G''}X''$ are equal, then 
the subgroups $H'$ and $H''$ of $H$ are equal; indeed, \cite{EF3}, Lemma 1.12 implies that $H'\subset H''$, which by symmetry 
implies the equality. Applying this to $_{G'}X'_{H'}$ and $_{G''}X''_{H''}$ being the subbitorsors of $\mathsf{G}^{\DR,\B}(\mathbf k)$ 
corresponding to $\mathsf{Stab}(\hat\Delta^{\mathcal M,\DR/\B})(\mathbf k)$ and $\mathsf{Stab}(\hat\Delta^{\mathcal W,\DR/\B})(\mathbf k)$ 
(see \cite{EF3}, Lemmas 3.6 and 3.9) using the equality of the underlying torsors yields the claimed equality. 
\end{proof}

\begin{cor}
The Lie algebra inclusion $\mathfrak{stab}(\hat\Delta^{\mathcal M,\B})\subset\mathfrak{stab}(\hat\Delta^{\mathcal W,\B})$ is an equality. 
\end{cor}

\begin{proof} 
The group schemes $\mathsf{Stab}(\hat\Delta^{\mathcal M,\B})$ and $\mathsf{Stab}(\hat\Delta^{\mathcal W,\B})$
are extensions of $\mathbb G_m$ by prounipotent group schemes. Theorem \ref{thm:24122021}, (c), then implies the equality 
of these group schemes, and therefore of their Lie algebras, which implies the statement. 
\end{proof}

\end{document}